\documentclass{amsart}
\usepackage{tikz}
\title[Toeplitz operators on large vector-valued Fock spaces]{Toeplitz operators on large vector-valued Fock spaces}

\author[H. Arroussi]{Hicham Arroussi}
\address{Department of Mathematics and Statistics, University of Reading, England}
\email{arruoussihicham@yahoo.fr}

\author[G. Asghari]{Ghazaleh Asghari}
\address{Department of Mathematics and Statistics, University of Reading, England}
\email{g.asgharikhonakdari@pgr.reading.ac.uk}

\author[J. Virtanen]{Jani Virtanen}
\address{Department of Physics and Mathematics, University of Eastern Finland\newline \indent Department of Mathematics and Statistics, University of Reading, England\newline \indent Department of Mathematics and Statistics, University of Helsinki, Finland}
\email{jani.virtanen@uef.fi, j.a.virtanen@reading.ac.uk, jani.virtanen@helsinki.fi}

\thanks{Arroussi was supported by the European Union’s Horizon 2022 research and innovation
programme under the Marie Sklodowska-Curie Grant Agreement No. 101109510. Asghari was supported by EPSRC grant EP/W524128/1. Virtanen was supported in part by EPSRC grant EP/Y008375/1.}

\usepackage{amsmath}
\usepackage{amssymb}
\usepackage{amsfonts}
\usepackage{amsthm}
\usepackage{mathtools}
\usepackage{enumerate}
\usepackage[toc,page]{appendix}
\usepackage[utf8]{inputenc}
\usepackage[
backend=biber,
sorting=nyt,
giveninits=true,
maxbibnames=99,
doi=false,
isbn=false
]{biblatex}
\usepackage{yfonts}
\usepackage{setspace}

\usepackage{tikz}
\usepackage[margin=1in]{geometry}
\usepackage{amsmath}

\usepackage{xcolor}
\usepackage{graphicx}
\DeclareMathOperator{\Tr}{Tr}

\theoremstyle{plain}
\newtheorem{thm}{Theorem}[section]

\newtheorem{lem}[thm]{Lemma}

\theoremstyle{definition}
\newtheorem{Rem}[thm]{Remark}
\newtheorem{defi}[thm]{Definition}

\usepackage{scalerel,stackengine}
\stackMath
\newcommand\reallywidehat[1]{
\savestack{\tmpbox}{\stretchto{
  \scaleto{
    \scalerel*[\widthof{\ensuremath{#1}}]{\kern-.6pt\bigwedge\kern-.6pt}
    {\rule[-\textheight/2]{1ex}{\textheight}}
  }{\textheight} 
}{0.5ex}}
\stackon[1pt]{#1}{\tmpbox}
}
\numberwithin{equation}{section}
 
\theoremstyle{plain}

\theoremstyle{definition}

\newtheorem{?}[Th]{Problem}

\usepackage[utf8]{inputenc}
\usepackage{amsmath}
\usepackage{amssymb}
\usepackage{amsfonts}
\usepackage{amsthm}
\usepackage{mathtools}
\usepackage{enumerate}
\usepackage[toc,page]{appendix}

\newcommand{\hH}{\mathcal{H}}

\newcommand{\C}{\mathbb{C}}
\usepackage{enumerate}

\begin{document}
\thispagestyle{empty}
  
\begin{abstract}
We characterize boundedness and compactness of Toeplitz operators on large vector-valued Fock spaces with Dall'Ara's weights [Adv.\ Math., 285 (2015) 1706--1740] in terms of generalized Berezin transforms, averaging functions, and Carleson measures. To determine Schatten class Toeplitz operators, we introduce the operator-valued Berezin transform and averaging functions.
\end{abstract}

\maketitle

\section{Introduction and Main Results}\label{introduction}

Let $\hH$ be a separable Hilbert space. We denote by $L^2_\phi(\hH)$ the space of all measurable $\hH$-valued functions on $\C^n$ for which
\begin{equation}\label{Fock space}
    \|f\|_{2,\phi}^2 = \int_{\C^n} \|f(z)\|^2_\hH \, e^{-2\phi(z)}\, dA(z) < \infty,
\end{equation}
where $dA$ is the Lebesgue measure on $\C^n$ and $\phi$ is an \emph{admissible weight} introduced by Dall'Ara~\cite{dall2015pointwise} (see \S\ref{Dall'Ara} below).
When equipped with the inner product
\begin{equation*}
    \langle f,g\rangle=\int_{\mathbb{C}^{n}} \langle f(z),g(z)\rangle_{\mathcal{H}}e^{-2\phi(z)}dA(z),
\end{equation*}
$L^2_\phi(\hH)$ becomes a Hilbert space. We say that $f:\C^{n}\to\mathcal{H}$ is holomorphic if for every continuous linear functional $\phi\in \mathcal{H}^{*}$, the scalar-valued function $\phi\circ f:\C^{n}\to\C$ is holomorphic in the usual sense (see, e.g., \S 3.10 in \cite{functionalAnalysisbook}). The \emph{vector-valued large Fock space} $F^2_\phi(\hH)$ is defined by
$$
    F^2_\phi(\hH) = L^2_\phi(\hH) \cap H(\C^n, \hH),
$$
where $H(\C^n, \hH)$ stands for the space of all $\hH$-valued holomorphic functions on $\C^n$. 

Our large Fock spaces $F^2_\phi(\mathcal{H})$ generalize the concept of doubling Fock spaces on $\C$ to higher dimensions and allow for vector-valued functions. Note that the class of admissible weights contains all weights $\phi$ for which there are constants $0<m<M$ such that
\begin{equation}\label{e:Kahler}
    m \omega_o \le dd^c \phi \le M\omega_o,
\end{equation}
where $\omega_o = \frac{1}{2}dd^c|z|^2$ is the Euclidean Kähler form in $\C^n$, $d=\partial+\bar{\partial}$ is the exterior derivative, and $d^c= \frac{i}{2}(\bar\partial-\partial)$; see, e.g., \cite{MR2891634} for further details of the weights satisfying \eqref{e:Kahler}. When $n=1$, the condition in \eqref{e:Kahler} is equivalent to $m\le \Delta \phi\le M$, where $\Delta$ is the Laplacian. 

It is easy to see that $F^2_\phi(\mathcal{H})$ is a closed subspace of $L^2_\phi(\mathcal{H})$ and hence a Hilbert space. Indeed, given $z\in\C^n$, by Lemma~\ref{VectorValued-Lemma2.4Hicham}, there is a constant $C(z)$ such that
\begin{equation*}
    \|f(z)\|_{\mathcal{H}}\leq C(z)\|f\|_{2,\phi}, \quad \textrm{for } f\in F^{2}_{\phi}(\mathcal{H})
\end{equation*}
(see Remark~\ref{remark2}), which implies that the point evaluation map $f\mapsto f(z)$ is a bounded linear homomorphism from $F^{2}_{\phi}(\mathcal{H})$ to $\mathcal{H}$ and uniformly bounded in bounded domains of $\C^n$. Since locally uniform limits of holomorphic functions are holomorphic, we conclude that $F^{2}_{\phi}(\mathcal{H})$ is a closed subspace of $L^{2}_{\phi}(\mathcal{H})$.

Reproducing kernel Hilbert spaces, such as the classical Fock space of square-integrable complex-valued holomorphic functions, have been an exciting area of research in analysis and operator theory. One of the basic properties of the reproducing kernel in the scalar setting, i.e, spaces of complex-valued functions, is that the reproducing kernel itself is holomorphic and belongs to the space. In the previous work on vector-valued Fock spaces, such as \cite{bigHankelBommierConstantine}, reproducing kernels were not explicitly defined and we introduce them here for the first time in Definition~\ref{RKHS}. However, a notion of \emph{operator-valued positive definite kernel}, introduced by Aronszajn in 1950 \cite{Aronszajn}, has been used in the finite-dimensional Euclidean setting in machine learning \cite{Pontil, ML}. Our definition agrees with this and also with the definition of \emph{matrix-valued reproducing kernel Hilbert spaces}; see e.g., \cite{matrixKernel}. What will be different, though, from the scalar case is that, although the reproducing kernel reproduces the elements of the Hilbert space in the sense of the integral equation \eqref{vectorKernel}, it is not an element of the space itself.

\begin{defi}\label{RKHS}
    Let $\mathcal{H}$ be a separable Hilbert space, $\mathcal{H}^{*}$ be its dual, and let $\mathcal{F}$ be a Hilbert space of functions $f:\C^{n}\to \mathcal{H}$. We say that $\mathcal{F}$ is a \emph{vector-valued reproducing kernel Hilbert} space if there is a map $K^{\mathcal{H}}:\C^{n}\times\C^{n}\to \mathcal{H}\otimes\mathcal{H}^{*}$ with $K^{\mathcal{H}}(z,w)^{*}\cong K^{\mathcal{H}}(w,z)$, and
    \begin{equation}\label{vectorKernel}
    f(z)=\int_{\C^{n}} K^{\mathcal{H}}(z,w)f(w)\, dV(w) \quad\textrm{for } f\in \mathcal{F},
\end{equation}
where $dV$ is a measure on $\C^n$.
\end{defi}

Note that here $\cong$ stands for the natural isomorphism $\mathcal{H}\otimes \mathcal{H}^{*}\cong \mathcal{H}^{*}\otimes \mathcal{H}$. Let $\mathcal{L}(\mathcal{H})$ be the set of bounded linear operators on $\mathcal{H}$. Then there is a natural isomorphism $\mathcal{L}(\mathcal{H})\cong \mathcal{H}^{*}\otimes \mathcal{H}$. In fact, using the map $B: \mathcal{H}^{*}\times \mathcal{H}\to \mathcal{L}(\mathcal{H})$ defined by $B(\lambda,w)(v)=\lambda(v)w$, and the universal property of the tensor products, we obtain a linear map $T_{B}:\mathcal{H}^{*}\otimes \mathcal{H}\to \mathcal{L}(\mathcal{H})$. This map is an isomorphism with inverse $S(L)=\sum_{i=1}^{\infty}e^{i}\otimes Le_{i}$, where $\{e_{i}\}_{i=1}^{\infty}$ is an orthonormal basis of $\mathcal{H}$ and $\{e^{i}\}_{i=1}^{\infty}$ is the dual basis of $\mathcal{H}^{*}$. Therefore, the vector-valued reproducing kernel $K^{\mathcal{H}}$ can be viewed as a map $K^{\mathcal{H}}:\C^{n}\times\C^{n}\to \mathcal{L}(\mathcal{H})$.

Write $K^{\mathcal{H}}_{z}(\cdot)=K^{\mathcal{H}}(\cdot,z)$. As a special case of Definition \ref{RKHS}, consider $K^{\mathcal{H}}:\C^{n}\times\C^{n}\to \mathcal{L}(\mathcal{H})$ and $K^{\mathcal{H}}(z,w)^{*}= K^{\mathcal{H}}(w,z)$. Let us consider the inner product of $\mathcal{F}$ as
\begin{equation*}
    \langle f,g\rangle_{\mathcal{F}}= \int_{\C^{n}}\langle f(z),g(z)\rangle_{\mathcal{H}}\, dV(z).
\end{equation*}
It follows that $\langle f(z),h\rangle_{\mathcal{H}}=\langle f, K^{\mathcal{H}}_{z}h\rangle_{\mathcal{F}}$ for every $h\in \mathcal{H}$ and $z\in\C^n$. This can be seen as an analog to $f(z)=\langle f, K_{z}\rangle$ in the scalar setting.

For the rest of the paper, we assume that $dV=e^{-2\phi}dA$, which implies that  $F^{2}_{\phi}(\mathcal{H})$ is a vector-valued reproducing kernel Hilbert space and its reproducing kernel $K^{\mathcal{H}}_{z}$ is a map from $\C^n$ to $\mathcal{H}\otimes\mathcal{H}^{*}$. The reproducing kernel property takes the form
\begin{equation*}
    f(z)=\int_{\C^{n}} K^{\mathcal{H}}(z,w)f(w)e^{-2\phi(w)}dA(w).
\end{equation*}
When $\mathcal{H}=\C$, we denote the scalar-valued weighted Fock space on $\C^n$ by $F^2_\phi(\C^n)$, which consists of all complex-valued holomorphic functions on $\C^n$ such that the norm defined in \eqref{Fock space} is finite.
Notice that the above integral is equivalent to the scalar reproducing kernel property, where the action of the reproducing kernel in the scalar case $F^{2}_{\phi}(\C^{n})$ is given by the usual multiplication. Being an element of $\mathcal{H}\otimes\mathcal{H}^{*}$, the most general $K^{\mathcal{H}}(z,w)$ is of the form $\sum_{m,n=1}^{\infty}K_{mn}(z,w)e_{m}\otimes e^{n}$, where $K_{mn}(z,w)$ are some complex scalars. In fact, we will see in \S\ref{vectorValuedRKSection} that the reproducing kernel for $F^{2}_{\phi}(\mathcal{H})$ is obtained by taking $K_{mn}(z,w)=\delta_{mn}K(z,w)$, where $K(z,w)$ is the reproducing kernel of $F^{2}_{\phi}(\C^{n})$; that is, 
\begin{equation*}
 K^{\mathcal{H}}_{w}(z)=K^{\mathcal{H}}(z,w)=\sum_{n=1}^{\infty}K(z,w)e_{n}\otimes e^{n}.
\end{equation*}

Define an integral operator $P:L^{2}_{\phi}(\mathcal{H})\to F^{2}_{\phi}(\mathcal{H})$ by 
\begin{equation}\label{e:OP}
   P(f)(z)=\int_{\C^{n}} K^{\mathcal{H}}(z,w)f(w)e^{-2\phi(w)}dA(w)=\int_{\C^{n}}f(w)
  K(z,w)e^{-2\phi(w)}dA(w),
\end{equation}
which is shown to be the orthogonal projection of $L^2_\phi(\mathcal{H})$ onto $F^2_\phi(\mathcal{H})$ in Lemma~\ref{Projection}.
To define vectorial Toeplitz operators, we denote by $T_\phi(\mathcal{L}(\mathcal{H}))$ the space of operator-valued functions $G:\mathbb{C}^{n}\to \mathcal{L}(\mathcal{H})$ such that each $G(z)$ is positive and satisfies 
\begin{equation}\label{KGL2}
K_z(\cdot)\|G(\cdot)\|_{\mathcal{L}(\mathcal{H})}\in L^2_\phi(\mathbb{C}^n), \quad z\in \mathbb{C}^n.
\end{equation}
For $G\in T_\phi(\mathcal{L}(\mathcal{H})),$ the \emph{vectorial Toeplitz operator} $T_{G}$ is defined by
\begin{equation*}
    T_{G}f(z)=P(Gf)(z)=\int_{\mathbb{C}^{n}}G(w)f(w)K(z,w)e^{-2\phi(w)}dA(w),
\end{equation*}
for $f\in F^{2}_{\phi}(\mathcal{H})$.

To characterize the boundedness and compactness of $T_G$, we define the Berezin transform $\tilde G$ by
$$
    \tilde G(z) = \int_{\mathbb{C}^{n}} |k_{z}(w)|^{2} e^{-2\phi(w)}\|G(w)\|_{\mathcal{L}(\mathcal{H})}dA(w), \quad z\in\mathbb{C}^{n},
$$
where 
$$
    k_{z}=\frac{K_{z}}{\|K_{z}\|_{F^{2}_{\phi}(\mathbb{C}^{n})}}, \quad z\in\C^n,
$$
is the normalized Bergman kernel of $F^{2}_{\phi}(\mathbb{C}^{n})$. For $r>0$, the corresponding averaging function $\widehat{G}_r$ is defined by 
$$
    \widehat{G}_r(z) = \frac{\int_{D^{r}(z)}\|G(w)\|_{\mathcal{L}(\mathcal{H})}dA(w)}{|D^{r}(z)|}
    \simeq
    \frac{\int_{D^{r}(z)}\|G(w)\|_{\mathcal{L}(\mathcal{H})}dA(w)}{\rho(z)^{2n}},
$$
where $|D^{r}(z)|$ is the Lebesgue measure of the disk $D^{r}(z)=D(z,r\rho(z))$ and $\simeq$ is defined below in \S\ref{Notation}.

We say that $G$ satisfies the Carleson condition if the inclusion map $I_{G}:F^{2}_{\phi}(\mathcal{H})\to L^{2}_{\phi}(\mathcal{H},\|G\|_{\mathcal{L}(\mathcal{H})}dA)$ is bounded, that is, there is a constant $C$ such that
\begin{align}\label{Carleson}
\left(\int_{\mathbb{C}^{n}} \|f(z)\|^{2}_{\mathcal{H}}e^{-2\phi(z)}\|G(z)\|_{\mathcal{L}(\mathcal{H})} dA(z) \right)^{1/2}\leq C \|f\|_{2,\phi}, \quad\textrm{for }f\in F^{2}_{\phi}(\mathcal{H}).
\end{align}

We say that $G$ satisfies the vanishing Carleson condition if the embedding operator $I_{G}:F^{2}_{\phi}(\mathcal{H})\to L^{2}_{\phi}(\mathcal{H},\|G\|_{\mathcal{L}(\mathcal{H})}dA)$ is compact, that is, for any bounded sequence $\{f_{j}\}_{j=1}^{\infty}$ in $F^{2}_{\phi}(\mathcal{H})$ that converges to zero uniformly on any compact subset of $\mathbb{C}^{n}$ as $j\to \infty$, 
\begin{align}\label{compact-Carleson}  
\lim_{j\to \infty}\int_{\mathbb{C}^{n}} \|f_{j}(z)\|^{2}_{\mathcal{H}}e^{-2\phi(z)}\|G(z)\|_{\mathcal{L}(\mathcal{H})} dA(z)=0.
\end{align}

 Vector-valued function theory, especially in the context of operator theory and functional analysis, has been a powerful tool in understanding scalar-valued function theory. For example, in several complex variables, scalar-valued holomorphic functions are special cases of sections of line bundles, and thus scalar problems are often special cases of richer vector bundle problems. In particular, extending a holomorphic function from a subdomain to a larger domain can be viewed as a special case of extending sections of vector bundles. So, one can lift the scalar case problem into a vector bundle setting, solve it there using the powerful Cartan's theorems A and B, and project the result back to the scalar world. For more details on Cartan's theorems and their applications, see, e.g., \S 7 in \cite{Hormander}. However, studying the vector-valued function theory for its own excitement is still justified. 

In the scalar setting, the basic properties of Toeplitz operators, such as boundedness, compactness, Schatten class membership and Fredholmness, are relatively well understood; see \cite{Jani, arroussi2022toeplitz}. Some of these results on boundedness and compactness have been generalized to the setting of vector-valued Fock spaces in \cite{bigHankelBommierConstantine, vector-valued2023positive}.

Our goal is first to obtain a better understanding of vector-valued Fock spaces $F^{2}_{\phi}(\mathcal{H})$, where $\phi$ is a generalization of doubling weights on the plane to $\C^{n}$. The difficulty lies behind the fact that these Fock spaces are reproducing kernel Hilbert spaces whose reproducing kernel does not belong to the space itself, violating the usual property of scalar Fock spaces. Next, we extend the previous characterizations and provide a complete description of boundedness, compactness, and Schatten class membership of vectorial Toeplitz operators acting on $F^{2}_{\phi}(\mathcal{H})$, in terms of the scalar and operator-valued Berezin transform and averaging functions. Accordingly, the proofs require new techniques adapted to such weights, such as generalized criteria of Carleson measures and decompositions of the complex plane by $r$-lattices, which differ from the conventional decompositions into subsets with essentially constant radius. Moreover, because of the lack of an explicit expression for the reproducing kernel, which makes our goal more difficult, we use some of its pointwise and norm estimates. All of these arguments play a crucial role in proving the main results. 

One of the most fundamental results about Toeplitz operators is that $T_{\psi}$ is bounded on the Hardy space if and only if the symbol $\psi$ is bounded. On the Bergman space, while a bounded symbol still defines a bounded Toeplitz operator, there are bounded Toeplitz operators with unbounded symbols. Studying Fock spaces, one can see that the boundedness of Toeplitz operators can be characterized by certain boundedness conditions on integrability and the Berezin transform of the measure, while compactness of Toeplitz operators is usually characterized by the vanishing of the symbol near the boundary or some decay at infinity. 

\subsection{Main results}

The following three results describe boundedness, compactness, and Schatten class properties of Toeplitz operators acting on large vector-valued Fock spaces.

\begin{thm}\label{Thm1.1}
Let $G\in T_\phi(\mathcal{L}(\mathcal{H}))$ and $\alpha$ be as in \eqref{2.7}. Then the following conditions are equivalent:
\begin{enumerate}[\rm(i)]
    \item $T_G : F^2_\phi(\mathcal{H})\to F^2_\phi(\mathcal{H})$ is bounded;
    \item $\widetilde G\in L^\infty(\C^n, dv)$;
    \item $\widehat{G}_\delta \in L^\infty(\C^n, dv)$ for some (or any) $0<\delta\le \alpha$;
    \item $\{\widehat{G}_\delta(z_k)\}_{k}$ is a bounded sequence for some (or any) $\delta$-lattice $\{z_k\}_{k}$ with $0<\delta\le \alpha$;
    \item $G$ satisfies a Carleson condition.
\end{enumerate}
 Moreover,
    \begin{equation}\label{th11n}
        \|T_{G}\|\simeq \| \Tilde{G}\|_{L^{\infty}(\mathbb{C}^{n},dA)}
        \simeq \| \hat{G}_{\delta}\|_{L^{\infty}(\mathbb{C}^{n},dA)}
        \simeq \|\{\hat{G}_{\delta}(z_{k})\}_{k}\|_{l^{\infty}}.
    \end{equation}
\end{thm}

\begin{thm}\label{Thm1.2}
Let $G\in T_\phi(\mathcal{L}(\mathcal{H}))$ and $\alpha$ be as in \eqref{2.7}. Then the following conditions are equivalent:
\begin{enumerate}[\rm(i)]
    \item $T_G : F^2_\phi(\mathcal{H})\to F^2_\phi(\mathcal{H})$ is compact;
    \item $\widetilde G(z) \to 0$ as $|z|\to \infty$;
    \item $\widehat{G}_\delta(z) \to 0$ as $|z|\to \infty$ for some (or any) $0<\delta\le \alpha$;
    \item $\widehat{G}_\delta(z_k) \to 0$ as $k\to\infty$ for some (or any) $\delta$-lattice $\{z_k\}_{k}$ with $0<\delta\le \alpha$;
    \item $G$ satisfies a vanishing Carleson condition.
\end{enumerate}
\end{thm}
To characterize the Schatten class membership of the vectorial Toeplitz operator $T_G$, we define the operator-valued Berezin transform of $G$ by
\begin{equation*}
   \Tilde{G}^{op}(z)
    =\int_{\mathbb{C}^{n}} |k_{z}(w)|^{2} e^{-2\phi(w)}\,G(w)\,dA(w), \quad z\in\mathbb{C}^{n},
\end{equation*}
and the corresponding averaging operator by 
\begin{equation*}
    \hat{G}^{op}_{r}(z)=\frac{\int_{D^{r}(z)}\,G(w)\, dA(w)}{|D^{r}(z)|}\simeq \frac{\int_{D^{r}(z)}G(w)dA(w)}{\rho(z)^{2n}}, \quad z\in\mathbb{C}^{n}.
\end{equation*}
These operator-valued versions of the Berezin transform and the averaging operator will likely be useful for the study of various classes of concrete operators. In our present work we use them to characterize the Schatten class membership of the vectorial Toeplitz operators.

\begin{thm}\label{thm1.7}
Let $0< p< \infty$, and $\delta<\min({1/2},\alpha)$, where $\alpha$ is as in (\ref{2.7}). There is an orthonormal basis $\{e_m^{z}\}_{m\geq 1}$ of $\mathcal{H}$, possibly depending on $z\in\C^{n}$, such that the following statements are equivalent:
\begin{enumerate}[\rm(i)]
    \item The operator $T_G$ belongs to $S_p(F^{2}_{\phi}(\mathcal{H}));$
    \item 
    \begin{equation*}
       \int_{\mathbb{C}^n}\sum_{m=1}^{\infty}\left(\langle \tilde{G}^{op}(z)\,e_m^{z},e_m^{z}\rangle_{\mathcal{H}}\right)^p\, \frac{dA(z)}{\rho(z)^{2n}}<\infty;
    \end{equation*}
    \item $$\int_{\mathbb{C}^n}\sum_{m=1}^{\infty}\left(\langle \hat{G}^{op}_{\delta}(z)\,e_m^{z},e_m^{z}\rangle_{\mathcal{H}}\right)^p\, \frac{dA(z)}{\rho(z)^{2n}}<\infty;$$
    \item Let $\{z_j\}_{j\geq 1}$ be a $\delta$-lattice. Then
    $$\sum_{j,m=1}^{\infty}\left(\langle \hat{G}^{op}_{\delta}(z_j)\,e_m^{j},e_m^{j}\rangle_{\mathcal{H}}\right)^p <\infty;$$    
\end{enumerate}
where $e_{m}^{j}=e^{z_{j}}_{m}$ is the basis of $\mathcal{H}$, obtained by eigenvectors of $\hat{G}^{op}_{\delta}(z_{j})$, for each $j\geq 1$. 
\end{thm}

Note that the integrals in the preceding theorem are taken with respect to the volume form associated with the Riemannian metric tensor $g=\sum_{j=1}^{n}\rho(z)^{-2}dz_{j}\otimes d\bar{z}_{j}$ over $\C^{n}$, taking into account the underlying geometry of the space. One can see that the associated Riemannian metric is conformal, with the conformal factor $\rho(z)^{-1}$.

 The paper is organized as follows. In Section \ref{preliminaries}, we give some background on the radius function $\rho$ and some useful estimates of the reproducing kernel. Further, we elaborate more on the relationship between the reproducing kernel $K^{\mathcal{H}}(z,w)$ of $F^{2}_{\phi}(\mathcal{H})$ and that of $F^{2}_{\phi}(\C^{n})$, and discuss the properties of the orthogonal projection. Furthermore, we provide some lemmas on the Schatten class properties of Toeplitz operators that turn out to be very useful in the proof of Theorem \ref{thm1.7}. Section \ref{Section3} is mostly devoted to the proof of Theorem \ref{Thm1.1} and Theorem \ref{Thm1.2}. Finally, the proof of Theorem \ref{thm1.7} is given in Section \ref{Section4}. 

\subsection{Notation}\label{Notation} We use $C$ to denote positive constants whose value may change from line to line but does not depend on the functions being considered. We say that $A\simeq B$ if there exists a constant $C>0$ such that $C^{-1}A\leq B\leq CA$. Moreover, $A\lesssim B$ if $A\leq CB$ for some positive constant $C$.

\section{Preliminaries}\label{preliminaries}
In this section, we state the definition of Dall'ara's weights, prove some key lemmas on the radius function $\rho$, and deal with the reproducing kernels of $F^{2}_{\phi}(\C^{n})$ and $F^{2}_{\phi}(\mathcal{H})$. We finish this section by providing some auxiliary results on the Schatten class membership of vectorial Toeplitz operators. 

\subsection{Dall'Ara's weights and the corresponding weighted Fock spaces}\label{Dall'Ara}

Let $\mathcal{H}$ be a separable Hilbert space with norm $\|\cdot\|_{\mathcal{H}}$ and $\phi:\mathbb{C}^{n}\to \mathbb{R}$ 
be a $\mathcal{C}^{2}$ plurisubharmonic function.
\begin{defi}\label{DefinitionDall'Ara}
We say that $\phi$ belongs to the weight class $\mathcal{W}$ if $\phi$ satisfies the following statements: 
\begin{itemize}
    \item[(I)] There exists $c>0$ such that
    \begin{equation}\label{phi}
        \inf_{z\in\mathbb{C}^{n}} \sup_{\xi\in D(z,c)}\Delta\phi(\xi)>0,
    \end{equation}
    where $D(z,c)$ is the Euclidean disk centered at $z$ with radius $c$,
    \item[(II)] $\Delta\phi$ satisfies the reverse-H\"older inequality. That is, there exists a positive real number $C$ such that
    \begin{equation*}
        \|\Delta\phi\|_{L^{\infty}(D(z,r))}\leq C r^{-2n}\int_{D(z,r)}\Delta\phi(\xi)dA(\xi), \quad\textrm{ for any }z\in\mathbb{C}^{n} \textrm{ and }r>0, 
    \end{equation*}
    \item[(III)]  the eigenvalues of  $H_{\phi}$  are comparable, i.e., there exists a $\delta_{0}>0$ such that 
    \begin{equation*}
        \langle H_{\phi}(z)u,u\rangle\geq \delta_{0}\Delta\phi(z)|u|^{2},\quad\textrm{ for any }u,z\in\mathbb{C}^{n},
    \end{equation*}
\end{itemize}
where The Hessian matrix  of $\phi$ is given by $$H_{\phi}= \left(\frac{\partial^2\phi}{\partial z_j\partial \overline{z_k}}\right)_{j,k\geq 1}.$$
\end{defi}
Suppose  $0<p<\infty$ and $\phi\in \mathcal{W}.$ The space $L^{p}_{\phi}(\mathbb{C}^{n})$ is the space of all measurable functions $f$ on $\mathbb{C}^{n}$ for which
\begin{equation*}
    \|f\|_{L^{p}_{\phi}(\mathbb{C}^{n})}=   \left(\int_{\mathbb{C}^{n}}|f(z)|^{p}e^{-p\phi(z)}dA(z) \right)^{1/p}<\infty,
\end{equation*}
and the space  $L^{\infty}_{\phi}(\mathbb{C}^{n})$  consists  of measurable functions endowed with the norm
\begin{equation*}
    \|f\|_{L^{\infty}_{\phi}(\mathbb{C}^{n})}=\sup_{z\in \mathbb{C}^{n}} |f(z)|e^{-\phi(z)}<\infty.
\end{equation*}
Denote by $H(\mathbb{C}^{n})$ the space of all holomorphic functions on $\mathbb{C}^{n}$. Then the scalar weighted Fock space is defined as
\begin{equation*}
    F^{p}_{\phi}(\mathbb{C}^{n})=L^{p}_{\phi}(\mathbb{C}^{n})\cap H(\mathbb{C}^{n}).
\end{equation*}
with  the same  norm which was defined above.  It is easy to check that $F^{p}_{\phi}(\mathbb{C}^{n})$ is a Banach space under the above norm for $1\le p<\infty$, and a complete metrizable topological vector space with the metric $$\varrho(f,g)= \|f-g\|_{F^{p}_{\phi}(\mathbb{C}^{n})}, \quad \textrm{ for }0<p < 1. $$
For $z\in\mathbb{C}^{n},$ we define the associated function $\rho$ to $\phi$ as 
\begin{equation}\label{radiusFuction}
    \rho(z)=\sup\{r>0: \sup_{w\in D(z,r)}\Delta\phi(w)\leq r^{-2}\}.
\end{equation}
This function satisfies many different properties, as presented in Lemma \ref{Lemmaofrhu}.
Let $\mu$ be a positive Borel measure defined by 
$$
    \mu(D(z,r))=r^{2}\|\Delta\phi\|_{L^{\infty}(D(z,r))}.
$$
One can see that $\mu$ is doubling, and $\mu(D(z,\rho(z)))=1$ using the reverse H\"older inequality, as  it was  shown in \cite{lv2017bergman}.

The orthogonal projection of $L^2_\phi(\C^n)$ onto $F^2_\Phi(\C^n)$ is denoted by $P_{\C}$ and the reproducing kernel of $F^{2}_{\phi}(\mathbb{C}^{n})$ by $K_{w}(z)=K(z,w)$. It is well known that $P_\C$ can represented as an integral operator 
\begin{equation*}
    P_{\C}(f)(z)=\int_{\mathbb{C}^{n}} f(w)K(z,w) e^{-2\phi(w)}dA(w), \quad z\in\mathbb{C}^{n},
\end{equation*}
which extends to a bounded projection from $L^p_\varphi(\C^n)$ to $F^p_\phi(\C^n)$ if $1<p<\infty$. In particular, Theorem 20 of \cite{dall2015pointwise} proves that there are constants
 $C,\varepsilon>0$ such that 
 \begin{equation}\label{pk}
|K(z,w)| \le C \frac{e^{\phi(z)} }{\rho(z)^n}\,\frac{e^{\phi(w)} }{\rho(w)^n}e^{-\varepsilon d_\rho(z,w)},\quad z,w\in \mathbb{C}^ n,
 \end{equation}
 where if $\gamma : [0,1]\to \mathbb{C}^n$ is a piecewise $C^1$ curve, we define   $$L_\rho(\gamma)=\int_{0}^{1} \frac{|\gamma'(t)|}{\rho(\gamma(t))} dt ,$$ 
  and $$d_\rho(z,w)=\inf_{\gamma} L_\rho(\gamma).$$ 
Moreover, $d_{\rho}(z,w)\simeq  \frac{|z-w|}{\rho(z)}$, for $z,w\in\C^{n}$. For more details please see proposition 5 in \cite{dall2015pointwise}.

\subsection{Some useful estimates}
The first lemma shows some properties of the associated function $\rho$ and construction of the $r$-lattice as defined below.
\begin{lem}[See \cite{WeightedCompositionHicham}, Lemma A] \label{Lemmaofrhu}
   Let $\phi$ be defined as in (\ref{phi}). Then the radius function $\rho$ satisfies the following properties.
   \begin{itemize}
       \item[(1)] There exists $M>0$ such that
       \begin{equation}\label{2.1}
           \sup_{z\in\mathbb{C}^{n}} \rho(z)\leq M,
       \end{equation}
       \item[(2)] The function $\rho$ is Lipschitz. That is, for every $z,w\in \mathbb{C}^{n}$,
       \begin{equation}\label{Lipschitz}
           |\rho(z)-\rho(w)|\leq |z-w|,
       \end{equation}
       \item[(3)] For $r\in (0,1)$ and $w\in D^{r}(z)$,
       \begin{equation}\label{2.2}
           (1-r)\rho(z)\leq \rho(w)\leq (1+r)\rho(z),
       \end{equation}
       \item[(4)] There exist $A,B>0$ such that
       \begin{equation}\label{2.3}
           |z|^{-A}\lesssim \rho(z)\lesssim |z|^{B}, \quad\textrm{for } |z|>1.
       \end{equation}
   \end{itemize}
\end{lem}
By (\ref{2.2}) and the triangle inequality, for any $r\in (0,1)$, there are $m_{1}=m_{1}(r)>1$ and $m_{2}=m_{2}(r)>1$ such that
\begin{equation}\label{2.4}
    D^{r}(z)\subset D^{m_{1}r}(w), \quad\textrm{and } D^{r}(w)\subset D^{m_{2}r}(z), \quad \textrm{for every }w\in D^{r}(z).
\end{equation}
It is easy to see that 
\begin{equation}\label{2.5}
    \beta=\sup_{0<r<1}[m_{1}(r)+m_{2}(r)]<\infty.
\end{equation}

Given a sequence $\{z_{k}\}_{k\geq 1}\subset\mathbb{C}^{n}$ and $r>0$, we call $\{z_{k}\}_{k\geq 1}$ an $r$-lattice if $\{D^{r}(z_{k})\}_{k=1}^{\infty}$ covers $\mathbb{C}^{n}$, and the balls of the form $\{D^{r/5}(z_{k})\}_{k=1}^{\infty}$ are pairwise disjoint. Moreover, for an $r$-lattice $\{z_{k}\}_{k\geq 1}$ and a real number $m\geq 1$, there exists some integer $N$, only depending on $m$ and $r$, such that each $z\in\mathbb{C}^{n}$ can be in at most $N$ balls of the form $D^{mr}(z_{k})$. That is, 
\begin{equation}\label{2.6}
    \sum_{k=1}^{\infty}\chi_{D^{mr}(z_{k})}(z)\leq N, \quad \textrm{ for } z\in \mathbb{C}^{n},
\end{equation}
where $\chi_{E}$ is a characteristic function of a subset $E$ of $\mathbb{C}^{n}$.

The second lemma presents some estimates of the reproducing kernels and their norms.
\begin{lem}[See \cite{arroussi2022toeplitz}, Lemma 2.3]\label{lem-Kernel estimates} Let $K_{z}=K(\cdot,z)$ be the reproducing kernel of $F^{2}_{\phi}(\mathbb{C}^{n})$. The following assertions are true.
\begin{itemize}
    \item[(a)] There exists $\alpha\in (0,1]$ such that
    \begin{equation}\label{2.7}
        |K_{z}(w)|\simeq \|K_{z}\|_{F^{2}_{\phi}(\mathbb{C}^{n})} \|K_{w}\|_{F^{2}_{\phi}(\mathbb{C}^{n})},\quad w\in D^{\alpha}(z),
    \end{equation}
\item[(b)] For $0<p\leq \infty$,
\begin{equation}\label{2.8}
    \|K_{z}\|_{F^{p}_{\phi}(\mathbb{C}^{n})} \simeq e^{\phi(z)}\rho(z)^{2n(1-p)/p}, \quad z\in \mathbb{C}^{n},
\end{equation}
\item[(c)] Let $\alpha$ be as defined in (\ref{2.7}). Then
\begin{equation}\label{2.10}
    |k_{z}(w)|^{2}e^{-2\phi(w)}\simeq \rho(z)^{-2n}, \quad w\in D^{\alpha}(z),
\end{equation}
\item[(d)] For each $z\in \mathbb{C}^{n}$, $0<p\leq \infty$ and $\beta\in \mathbb{R}$, 
\begin{equation}\label{kernelIntegral}
    \int_{\mathbb{C}^{n}}|K_{z}(w)|^{p}e^{-p\phi(w)}\rho(w)^{\beta}dA(w)\simeq
    e^{p\phi(z)}\rho(z)^{2n(1-p)+\beta},
\end{equation}
\item[(e)] The set $\{k_{z}:z\in\mathbb{C}^{n}\}$ is bounded in $F^{2}_{\phi}(\mathbb{C}^{n})$ and $k_{z}\to 0$ uniformly on any compact subsets of $\mathbb{C}^{n}$ as $|z|\to \infty$.  
\end{itemize}    
\end{lem}

The third lemma shows how we transform any function $f$ in $ F^{2}_{\phi}(\mathcal{H})$ into another one in $F^{2}_{\phi}(\mathbb{C}^{n})$.
\begin{lem}\label{Lemma2.3}
    If $f\in F^{2}_{\phi}(\mathcal{H})$, then $z\mapsto \langle  
f(z),e \rangle_{\mathcal{H}}\in F^{2}_{\phi}(\mathbb{C}^{n})$, for any unit element $e\in \mathcal{H}$.
\end{lem}
\begin{proof}
    By Cauchy-Schwarz inequality,  
    $$
    \int_{\mathbb{C}^{n}}|\langle f(z),e\rangle_{\mathcal{H}}|^{2} e^{-2\phi(z)}dA(z)\leq \int_{\mathbb{C}^{n}}\|f(z)\|^{2}_{\mathcal{H}}\,\|e\|^{2}_{\mathcal{H}}\,e^{-2\phi(z)}dA(z)<\infty,$$
    which finishes the proof.
\end{proof}
\begin{Rem}
    Lemma \ref{Lemma2.3} can be generalized for any $h\in H$. Indeed, similarly, one can observe that $z\mapsto \langle  
f(z),h \rangle_{\mathcal{H}}\in F^{2}_{\phi}(\mathbb{C}^{n})$, for any element $h\in \mathcal{H}$.
\end{Rem}
\begin{lem}\label{WeakConvergenceNormalizedKernel}
    The set $\{k_{z}(\cdot)e:z\in\mathbb{C}^{n}\}$ is bounded in $F^{2}_{\phi}(\mathcal{H})$ and $k_{z}(\cdot)e\to 0$ uniformly on any compact subsets of $\mathbb{C}^{n}$ as $|z|\to \infty$.  
\end{lem}
\begin{proof}
 It is similar to the proof of the statement (e) of Lemma \ref{lem-Kernel estimates}.
 \end{proof}
\begin{lem}[See \cite{WeightedCompositionHicham}, Lemma B]\label{Lemma2.4-HichamToeplitz}
    Let $0<p<\infty$ and define $\phi$ as in (\ref{phi}). For any $\delta\in (0,1]$, there exists $C>0$ such that for any $f\in H(\mathbb{C}^{n})$ and $z\in \mathbb{C}^{n}$,
    \begin{equation}
        |f(z)|^{p}e^{-p\phi(z)}\leq \frac{C}{\delta^{2n}\rho(z)^{2n}}\int_{D^{\delta}(z)} |f(w)|^{p}e^{-p\phi(w)}dA(w).
    \end{equation}
\end{lem}
\begin{lem}\label{VectorValued-Lemma2.4Hicham} 
     For any $\delta\in (0,1]$, there exists $C>0$ such that for any $f\in F^{2}_{\phi}(\mathcal{H})$ and $z\in \mathbb{C}^{n}$,
\begin{equation}\label{VectorLemma2.4-Hicham}
    \|f(z)\|^{2}_{\mathcal{H}}e^{-2\phi(z)}\leq \frac{C}{\delta^{2n}\rho(z)^{2n}}\int_{D^{\delta}(z)} \|f(w)\|^{2}_{\mathcal{H}}e^{-2\phi(w)} dA(w).
    \end{equation}
\end{lem}
\begin{proof}
    Let $f\in F^{2}_{\phi}(\mathcal{H})$. By Lemma \ref{Lemma2.3}, $\langle f(z),e\rangle_{\mathcal{H}}$ belongs to $F^{2}_{\phi}(\mathbb{C}^{n})$ and hence holomorphic, for any unit vector $e\in \mathcal{H}$. Hence by Lemma \ref{Lemma2.4-HichamToeplitz}, and applying the Cauchy-Schwarz inequality
    \begin{align*}
        |\langle f(z),e\rangle_{\mathcal{H}}|^{2}e^{-2\phi(z)}&\leq \frac{C}{\delta^{2n}\rho(z)^{2n}}\int_{D^{\delta}(z)} |\langle f(w),e\rangle_{\mathcal{H}}|^{2}e^{-2\phi(w)}dA(w)\\
        & \leq\frac{C}{\delta^{2n}\rho(z)^{2n}}\int_{D^{\delta}(z)} \| f(w)\|_{\mathcal{H}}^{2}\|e\|^{2}_{\mathcal{H}}e^{-2\phi(w)}dA(w).
    \end{align*}
    Since $\|e\|_{\mathcal{H}}=1,$ we obtain (\ref{VectorLemma2.4-Hicham}) and the proof is complete.
\end{proof}
\begin{Rem}\label{remark2}
Let $z\in \C^n$. Then by Lemma \ref{VectorValued-Lemma2.4Hicham}, $\|f(z)\|_{\mathcal{H}}\leq C\frac{e^{\phi(z)}}{\rho(z)^{n}}\|f\|_{2,\phi}$, and hence the point evaluation map $f\mapsto f(z)$ is a bounded linear homomorphism from $F^{2}_{\phi}(\mathcal{H})$ to $\mathcal{H}$. Let $C(z)$ be the bounding constant, depending only on $z$, $\phi$, and $n$. One can see that for any compact subset $K\in \C^{n}$, and any $z\in K$, $C(z)$ is bounded. To see this, first take $K$ not overlapping the unit disk centered at the origin. Then \eqref{2.3} implies that $C(z)\simeq e^{\phi(z)}|z|^{nA}$ for some $A>0$, and thus bounded. Now, assume that $K$ overlaps the unit disk $D(0,1)$. By \eqref{Lipschitz}, $\rho$ is continuous, and since $\phi$ is $\mathcal{C}^{2}$, it is enough to show that $\rho$ never vanishes on the unit disk, to conclude that $C(z)$ is continuous and thus bounded on $K$. Let $z\in D(0,1)$. By continuity of $\Delta\phi$, and since $\phi$ is plurisubharmonic, there is some constant $M>0$ such that $\sup_{w\in D(0,2)}\Delta\phi(w)=M$. Let $N=\max \{M,1\}$. Then $\frac{1}{N}<1$, and thus $D(z,\frac{1}{N})\subset D(0,2)$, for every $z\in D(0,1)$. Hence, $\sup_{w\in D(z,\frac{1}{N})}\Delta\phi(w)\leq N\leq N^{2}$. Using \eqref{radiusFuction}, we can conclude that $\rho(z)\geq \frac{1}{N}$ for every $z\in D(0,1)$, and in particular $\rho(z)\neq 0$.  
\end{Rem}

\subsection{Reproducing kernel of $F^{2}_{\phi}(\mathcal{H})$ and the orthogonal projection}\label{vectorValuedRKSection}
\begin{lem}\label{RKP}
Let $\phi$ be as in Definition \ref{DefinitionDall'Ara}, and $\mathcal{H}$ be a separable Hilbert space. The reproducing kernel of $F^{2}_{\phi}(\mathcal{H})$ is of the form 
\begin{equation*}
 K^{\mathcal{H}}_{w}(z)=K^{\mathcal{H}}(z,w)=\sum_{n=1}^{\infty}K(z,w)e_{n}\otimes e^{n},
\end{equation*}
where $K(z,w)$ is the reproducing kernel of $F^{2}_{\phi}(\C^{n})$.
\end{lem}
\begin{proof}
Applying Lemma \ref{Lemma2.3}, we can write
\begin{align*}
 &\int_{\C^{n}} K^{\mathcal{H}}(z,w)f(w)e^{-2\phi(w)}dA(w)=\\
 &\int_{\C^{n}}\sum_{n=1}^{\infty} \langle f(w),e_{n}\rangle_{\mathcal{H}}e_{n}
  K(z,w)e^{-2\phi(w)}dA(w)\\
 & =\sum_{n=1}^{\infty} \langle f(z),e_{n}\rangle_{\mathcal{H}}e_{n}\\
  &=f(z),
\end{align*}
showing that the choice we made for the reproducing kernel does make sense. Moreover, since $K(z,w)$ is conjugate symmetric, and by the natural isomorphism $\mathcal{H}\otimes \mathcal{H}^{*}\cong \mathcal{H}^{*}\otimes \mathcal{H}$,
\begin{equation*}
  K^{\mathcal{H}}(z,w)^{*}=\sum_{n=1}^{\infty}\overline{K(z,w)}e^{n}\otimes e_{n} \cong\sum_{n=1}^{\infty}K(w,z)e_{n}\otimes e^{n} =K^{\mathcal{H}}(w,z).
\end{equation*}
\end{proof}

We now show that the integral operator defined in \eqref{e:OP} is the orthogonal projection onto $F^2_\phi(\mathcal{H})$.

\begin{lem}\label{Projection}
Let $\phi$ be as in Definition \ref{DefinitionDall'Ara}, and $\mathcal{H}$ be a separable Hilbert space. The integral operator
\begin{align*}
   & P(f)(z)=\int_{\C^{n}} K^{\mathcal{H}}(z,w)f(w)e^{-2\phi(w)}dA(w)=\int_{\C^{n}}f(w)
  K(z,w)e^{-2\phi(w)}dA(w), \quad z\in\C^n,
\end{align*}
is the orthogonal projection of $L^2_\phi(\mathcal{H})$ onto $F^2_\phi(\mathcal{H})$.
\end{lem}
\begin{proof}
    For $f\in L^{2}_{\phi}(\mathcal{H})$, take $f^{*}(z)=\langle f(z),e_{n}\rangle
_{\mathcal{H}}$, and recall the scalar orthogonal projection $P_{\C}:L^{2}_{\phi}(\C^{n})\to F^{2}_{\phi}(\C^{n})$. By Lemma \ref{RKP}, it is immediate to see that
\begin{equation*}
    P(f)(z)=\sum_{n=1}^{\infty}e_{n}P_{\C}f^{*}(z), \quad f\in L^{2}_{\phi}(\mathcal{H}),
\end{equation*}
is an element of $F^{2}_{\phi}(\mathcal{H})$. Furthermore, $P\circ P=P$. Hence, $P$ is a projection and, in fact, it is orthogonal. To see this, for $f\in F^{2}_{\phi}(\mathcal{H})$, write $f=(f-Pf)+Pf$. One can show that $\langle f-Pf,Pf\rangle =0$. Indeed,
\begin{align*}
    \langle f,Pf\rangle &=\int_{\C^{n}}\langle f(z),Pf(z)\rangle_{\mathcal{H}}e^{-2\phi(z)}dA(z)\\
    &=\int_{\C^{n}}\langle \sum_{m=1}^{\infty}e_{m}\langle f(z),e_{m}\rangle_{\mathcal{
    H}}, \sum_{k=1}^{\infty}e_{k}P_{\C}f^{*}(z)\rangle_{\mathcal{H}}e^{-2\phi(z)}dA(z)\\
    &=\int_{\C^{n}} \sum_{m=1}^{\infty}\langle f(z),e_{m}\rangle_{\mathcal{H}} \overline{P_{\C}f^{*}(z)}e^{-2\phi(z)}dA(z),
\end{align*}
and
\begin{equation*}
\langle Pf,Pf\rangle =  
\int_{\C^{n}} \sum_{m=1}^{\infty}P_{\C}f^{*}(z) \overline{P_{\C}f^{*}(z)}e^{-2\phi(z)}dA(z),
\end{equation*}
implying that
\begin{align*}
 \langle f,Pf\rangle- \langle Pf,Pf\rangle =
 \sum_{m=1}^{\infty}\left[ 
 \langle f^{*},P_{\C}f^{*}\rangle_{F^{2}_{\phi}(\C^{n})}-  \langle P_{\C}f^{*},P_{\C}f^{*}\rangle_{F^{2}_{\phi}(\C^{n})}
 \right]=0,
\end{align*}
since $P_{\C}:L^{2}_{\phi}(\C^{n})\to F^{2}_{\phi}(\C^{n})$ is an orthogonal projection.

\end{proof}

\subsection{Schatten classes} In this subsection, we prove some lemmas that will be useful in the proof of Theorem~\ref{thm1.7}. 

\begin{lem}\label{lem2.14}
 Let $\{e_k^{z}\}_{k\geq 1}$ be an orthonormal basis of $\mathcal{H}$, possibly depending on $z\in\C^{n}$, and assume that $1\leq p\leq \infty$.
Then, $T_G\in S_p(F^{2}_{\phi}(\mathcal{H}))$ if 
$$\int_{\mathbb{C}^n}\sum_{m=1}^{\infty}\left(\langle G(z)\,e_m^{z},e_m^{z}\rangle \right)^p\, \frac{dA(z)}{\rho(z)^{2n}}<\infty.$$
\end{lem}
\begin{proof}
Let $\{f_k\}_{k\geq 1}$ be an orthonormal basis of $F^{2}_{\phi}(\mathbb{C}^{n})$. Consider $\{B_{k,m}(z)= f_k(z) e_m^{z}\}_{m,k\geq 1}$, which is an orthonormal basis of $F^{2}_{\phi}(\mathcal{H}).$ First, by the reproducing kernel property, it is easy to see that for any $f,g\in F^{2}_{\phi}(\mathcal{H})$,
\begin{align*}
    \langle T_{G}f,g\rangle&=\int_{\C^{n}}\langle T_{G}f(z),g(z)\rangle_{\mathcal{H}}e^{-2\phi(z)}dA(z)
    \\&=\int_{\C^{n}}\int_{\C^{n}}\langle G(w)f(w),g(z)\rangle_{\mathcal{H}}K(z,w)e^{-2\phi(w)}e^{-2\phi(z)}dA(w)dA(z)\\
    &=\int_{\C^{n}}\langle G(w)f(w),g(w)\rangle_{\mathcal{H}}e^{-2\phi(w)}dA(w).
\end{align*}
Hence,
\begin{align*}
\langle T_{G}B_{k,m},B_{k,m}\rangle &= \int_{\mathbb{C}^n} \langle T_{G}f_k(z) e_m^{z},f_k(z) e_m^{z}\rangle_{\mathcal{H}} \,e^{-2\phi(z)}\,dA(z)\\
&=\int_{\mathbb{C}^n} \langle G(z)e_m^{z},e_m^{z}\rangle_{\mathcal{H}} |f_k(z)|^2 \,e^{-2\phi(z)}\,dA(z).
\end{align*}
Hence, for $p=\infty,$
$$
\|T_{G}\|_{S_{\infty}(F^{2}_{\phi}(\mathcal{H}))}\lesssim \sup_{z\in\mathbb{C}^n}\sum_{m=1}^{\infty}\langle G(z)e_m^{z},e_m^{z}\rangle_{\mathcal{H}}.
$$
On the other hand, for $p=1,$ applying (\ref{2.8}), and $ K_z(z)=\sum_{k=1}^{\infty}|f_k(z)|^2 $,  we have 
\begin{align*}
    \sum_{m=1}^{\infty} \sum_{k=1}^{\infty}\langle T_{G}B_{k,m},B_{k,m}\rangle &=\int_{\mathbb{C}^n} \sum_{m=1}^{\infty}\langle G(z)e_m^{z},e_m^{z}\rangle_{\mathcal{H}} \left(\sum_{k=1}^{\infty}|f_k(z)|^2 \right)\,e^{-2\phi(z)}\,dA(z)\\
    &= \int_{\mathbb{C}^n} \sum_{m=1}^{\infty}\langle G(z)e_m^{z},e_m^{z}\rangle_{\mathcal{H}} K_z(z)\,e^{-2\phi(z)}\,dA(z)\\
    &\simeq \int_{\mathbb{C}^n} \sum_{m=1}^{\infty}\langle G(z)e_m^{z},e_m^{z}\rangle_{\mathcal{H}} \,\frac{dA(z)}{\rho(z)^{2n}}.
\end{align*}
Then by Proposition 1.29 in \cite{ZhuOperatorTheory},
$$
\|T_{G}\|_{S_{1}(F^{2}_{\phi}(\mathcal{H}))}\lesssim \int_{\mathbb{C}^n}\sum_{m=1}^{\infty}\langle G(z)e_m,e_m\rangle_{\mathcal{H}}\,\frac{dA(z)}{\rho(z)^{2n}}.
$$
By the interpolation, we can obtain the desired result, which completes the proof.
\end{proof}
\begin{lem}\label{compactnessOfG-hat}
Let $T_{G}:F^{2}_{\phi}(\mathcal{H})\to F^{2}_{\phi}(\mathcal{H})$ be compact and assume that $\delta>0$. Then for every $z\in\C^{n}$, the average operator $\hat{G}^{op}_{\delta}(z)\simeq\int_{D^{\delta}(z)}G(w)\frac{dA(w)}{\rho(w)^{2n}}:\mathcal{H}\to \mathcal{H}$ is also compact.
\end{lem}
\begin{proof}
   Fix $z\in\C^{n}$. Using (\ref{2.2}), (\ref{2.8}), and (\ref{2.10}) we can write
    \begin{equation*}
        \hat{G}^{op}_{\delta}(z)=C(z)\int_{D^{\delta}(z)}G(w)K(z,w)K(w,z)e^{-2\phi(w)}dA(w),
    \end{equation*}
where $C(z)=\frac{\rho(z)^{2n}}{e^{2\phi(z)}}$. Let $h\in\mathcal{H}$ be arbitrary and $f=K_{z}(\cdot)h$ be an element of $F^{2}_{\phi}(\mathcal{H})$. Then 
\begin{align*}
   &C(z) T_{G}\left(K_{z}(\cdot)h\right)(z)= C(z)\int_{\C^{n}}G(w)[hK(w,z)]K(z,w)e^{-2\phi(w)}dA(w)\\
    &\quad\quad\quad= \hat{G}^{op}_{\delta}(z)h+C(z)\int_{\C^{n}\setminus D^{\delta}(z)}G(w)[hK(z,w)]K(w,z)e^{-2\phi(w)}dA(w).
\end{align*}
  Let $\{x_{n}\}_{n\geq 1}$ be a sequence in $\mathcal{H}$, converging weakly to zero. Then $\{K_{z}(\cdot)x_{n}\}_{n\geq 1}$ is a sequence in $F^{2}_{\phi}(\mathcal{H})$, that converges weakly to zero. Since $T_{G}$ is compact, $\|T_{G}\left(K_{z}(\cdot)x_{n}\right)\|^{2}_{2,\phi }\to 0$. Then Lemma \ref{VectorValued-Lemma2.4Hicham} implies that $\|T_{G}\left(K_{z}(\cdot)x_{n}\right)(z)\|^{2}_{\mathcal{H} }\to 0$ for every $z\in\C^{n}$. For each $n\geq 1$ define
  \begin{equation*}
      A_n:=\hat{G}^{op}_{\delta}(z)x_{n},
  \end{equation*}
  and
  \begin{equation*}
      B_{n}:= C(z)\int_{\C^{n}\setminus D^{\delta}(z)}G(w)[x_{n}K(z,w)]K(w,z)e^{-2\phi(w)}dA(w).
  \end{equation*}
  Using $x_{n}\rightharpoonup 0$ as $n\to\infty$, we can see that for each $w,\xi\in \C^{n}$, $\langle G(w)x_{n}, G(\xi)x_{n}\rangle_{\mathcal{H}}\to 0$ and $n\to \infty$. Therefore, $\langle A_{n},B_{n}\rangle_{\mathcal{H}}\to 0$ as $n\to \infty$. Thus, by Parseval's identity, $C(z)^{2}\|T_{G}\left(K_{z}(\cdot)x_{n}\right)(z)\|^{2}_{\mathcal{H} }\geq \|\hat{G}^{op}_{\delta}(z)x_{n}\|^{2}_{\mathcal{H}}$, as $n\to \infty$. Since $C(z)$ is a real number for each $z\in\C^{n}$, we can conclude that as $n\to \infty$, $\|\hat{G}^{op}_{\delta}(z)x_{n}\|^{2}_{\mathcal{H}}\to 0$, implying that $\hat{G}^{op}_{\delta}(z)$ is compact.
    \end{proof}

\section{Boundedness and Compactness of Vectorial Toeplitz operators}\label{Section3}
 In this section, we prove Theorem \ref{Thm1.1} and Theorem \ref{Thm1.2}, to characterize boundedness and compactness of the vectorial Toeplitz operator $T_{G}: F^{2}_{\phi}(\mathcal{H})\to F^{2}_{\phi}(\mathcal{H})$.
\begin{proof}[Proof of Theorem \ref{Thm1.1}] 
First, we prove that (ii) implies (iii). By using  \eqref{2.10}, we obtain
\begin{align*}
    \hat{G}_{\delta}(z)&\simeq  \rho(z)^{-2n}\int_{D^{\delta}(z)}\|G(w)\|_{\mathcal{L}(\mathcal{H})}dA(w)\\
    & \simeq \int_{D^{\delta}(z)} |k_{z}(w)|^{2} e^{-2\phi(w)} \|G(w)\|_{\mathcal{L}(\mathcal{H})}dA(w)\\
    &\leq\Tilde{G}(z),
\end{align*}
 which gives (iii). Moreover,
\begin{equation}\label{ineq3.1}
    \|\hat{G}_{\delta}\|_{L^{\infty}(\mathbb{C}^{n},dA)}\lesssim \|\Tilde{G}\|_{L^{\infty}(\mathbb{C}^{n},dA)}
\end{equation}
It is obvious that (iii) implies (iv) and 
\begin{equation}\label{ineq3.2}
\|\lbrace\hat{G}_{\delta}(z_k)\rbrace_k\|_{\ell^{\infty}}\lesssim\|\hat{G}_{\delta}\|_{L^{\infty}(\mathbb{C}^{n},dA)}.
\end{equation}
Now, let us show that (iv) implies (ii). Note that
\begin{align}\label{G-tilde-s}
    \Tilde{G}(z)&=\int_{\mathbb{C}^{n}} |k_{z}(w)|^{2} e^{-2\phi(w)} \|G(w)\|_{\mathcal{L}(\mathcal{H})} dA(w)\nonumber\\
    & \leq\sum_{k=1}^{\infty}\int_{D^{\delta}(z_{k})}|k_{z}(w)|^{2} e^{-2\phi(w)} \|G(w)\|_{\mathcal{L}(\mathcal{H})} dA(w)\nonumber\\
    &\lesssim \sum_{k=1}^{\infty} \hat{G}_{\delta}(z_k) \rho(z_k)^{2n}  \sup_{w\in D^{\delta}(z_{k})} |k_{z}(w)|^{2} e^{-2\phi(w)}.
\end{align}
By Lemma \ref{Lemma2.4-HichamToeplitz} and  taking a real number 
$m>1$, as well as (\ref
{2.6}), we obtain 
\begin{align*}
    \Tilde{G}(z)
    &\lesssim \sum_{k=1}^{\infty} \hat{G}_{\delta}(z_{k}) \int_{D^{m\delta}(z_{k})}|k_{z}(w)|^{2}e^{-2\phi(w)}dA(w) \\
    &\leq \sup_{k} \hat{G}_{\delta}(z_{k})  \sum_{k=1}^{\infty}\int_{D^{m\delta}(z_{k})}|k_{z}(w)|^{2}e^{-2\phi(w)}dA(w) \\
    & \leq N\sup_{k} \hat{G}_{\delta}(z_{k})  \|k_{z}\|^{2}_{F^{2}_{\phi}(\C^{n})}\\
    &\lesssim 
      \sup_{k} \hat{G}_{\delta}(z_{k}).
\end{align*}
Therefore,
 \begin{equation}\label{ineq3.4}
     \|\Tilde{G} \|_{L^{\infty}(\mathbb{C}^{n},dA)}\lesssim 
     \|\{ \hat{G}_{\delta}(z_{k}) \}_{k}\|_{l^{\infty}}.
 \end{equation}
Thus, (ii), (iii) and (iv) are all equivalent.
 
Next we show that (i) implies (ii). Let $T_{G}:F^{2}_{\phi}(\mathcal{H})\to F^{2}_{\phi}(\mathcal{H})$ be bounded. Since $G(w)$ is positive for every $w\in\C^{n}$, $\|G(w)\|_{\mathcal{L}(\mathcal{H})}=\sup_{\|e\|=1}\langle G(w)e,e\rangle_{\mathcal{H}}$. Applying \eqref{2.8}, Lemma \ref{WeakConvergenceNormalizedKernel}, and Lemma \ref{VectorValued-Lemma2.4Hicham}, we obtain  
\begin{equation}\label{compactTG}
\begin{split}
    \Tilde{G}(z)&=\int_{\mathbb{C}^{n}}|k_{z}(w)|^{2}e^{-2\phi(w)} \|G(w)\|_{\mathcal{L}(\mathcal{H})}dA(w)\\
    &\simeq 
    \rho(z)^{n}e^{-\phi(z)}\int_{\mathbb{C}^{n}}k_{z}(w)K(z,w)e^{-2\phi(w)} \|G(w)\|_{\mathcal{L}(\mathcal{H})}dA(w)\\
     &\simeq\rho(z)^{n}e^{-\phi(z)} \int_{\mathbb{C}^{n}} \sup_{\|e\|=1}\langle G(w)k_{z}(w)e,e  \rangle_{\mathcal{H}} K(z,w)e^{-2\phi(w)} dA(w)\\
&\simeq \rho(z)^{n} e^{-\phi(z)}\sup_{\|e\|=1}  \langle T_{G}k_{z}(z)e,e  
    \rangle_{\mathcal{H}} \\
    &\leq\rho(z)^{n} e^{-\phi(z)} \|T_{G}(k_{z}(z)e)\|_{\mathcal{H}}\\
    &\lesssim \rho(z)^{n}\left(\frac{1}{\rho(z)^{2n}}\int_{\mathbb{C}^n}\|T_{G}k_{z}(\zeta)e\|_{\mathcal{H}}^{2}e^{-2\phi(\zeta)} dA(\zeta) \right)^{1/2}\\
    &\lesssim \|T_{G}k_{z}(\cdot)e\|_{2,\varphi}\lesssim \|T_{G}\|. 
\end{split}    
\end{equation}  

To show that (iii) implies (i), we claim that there is a constant $C>0$ such that
\begin{equation}\label{star}
    \|T_{G}f\|^{2}_{2,\phi}\leq C\int_{\mathbb{C}^{n}} \|f(w)\|^{2}_{\mathcal{H}} e^{-2\phi(w)}\hat{G}_{\delta}(w)^{2}dA(w),
\end{equation}
for any $f\in F^{2}_{\phi}(\mathcal{H})$ and any $\delta>0.$  Then \eqref{star} and \eqref{VectorLemma2.4-Hicham} imply that
\begin{align*}\label{P2.3}
\|T_{G}f\|^{2}_{2,\phi}&
    \lesssim \int_{\mathbb{C}^{n}}  \hat{G}_{\delta}(w)^{2}
   \|f(w)\|^{2}_{\mathcal{H}} e^{-2\phi(w)}dA(w)\nonumber\\
   & \leq \|\hat{G}_{\delta}\|^{2}_{L^{\infty}(\mathbb{C}^{n},dA)} \|f\|^{2}_{2,\phi}.
\end{align*}
Hence, $T_{G}: F^{2}_{\phi}(\mathcal{H})\to F^{2}_{\phi}(\mathcal{H})$ is bounded and 
\begin{equation}\label{ineq3.7}
    \|T_{G}\| \lesssim \|\hat{G}_{\delta}\|_{L^{\infty}(\mathbb{C}^{n},dA)}.
\end{equation}
To finish the proof, we should justify the inequality \eqref{star}. Let $f\in F^{2}_{\phi}(\mathcal{H})$ and $z\in \mathbb{C}^{n}$. Take $r>0$ such that $\beta^{2} r\leq \delta$, where $\beta$ is as in (\ref{2.5}), and let $\{z_{k}\}_{k}$ be an $r$-lattice. Then

\begin{align*}
    \|T_{G}(f)(z)\|_{\mathcal{H}}&=\left\|
    \int_{\mathbb{C}^{n}} G(w)f(w)K(z,w) e^{-2\phi(w)}dA(w)
    \right\|_{\mathcal{H}}\\
    &\leq  \int_{\mathbb{C}^{n}} \|G(w)\|_{\mathcal{L}(\mathcal{H})}\|f(w)\|_{\mathcal{H}}|K(z,w)| e^{-2\phi(w)}dA(w)\\
    &\leq \sum_{k=1}^{\infty} \int_{D^{r}(z_{k})} \|G(w)\|_{\mathcal{L}(\mathcal{H})}\|f(w)\|_{\mathcal{H}}|K(z,w)| e^{-2\phi(w)}dA(w)\\
     &\lesssim \sum_{k=1}^{\infty}\hat{G}_{r}(z_{k})\rho(z_{k})^{2n}\left( \sup_{w\in D^{r}(z_{k})}\|f(w)\|_{\mathcal{H}}|K(z,w)| e^{-2\phi(w)}\right).
\end{align*}
Applying \eqref{VectorLemma2.4-Hicham}, since the function $fK_z\in H(\mathbb{C}^n, \mathcal{H}),$  using $\hat{G}_{r}(z_{k})\lesssim \hat{G}_{\beta^{2}r}(w),$ for any $w\in D^{\beta r}(z_{k}),$ and \eqref{2.2}, we obtain
\begin{align}\label{3.11}
    \|T_{G}(f)(z)\|_{\mathcal{H}}
    &\lesssim 
   \sum_{k=1}^{\infty}\hat{G}_{r}(z_{k})
   \int_{D^{\beta r}(z_{k})}\|f(w)\|_{\mathcal{H}}|K(z,w)| e^{-2\phi(w)}dA(w)\nonumber\\
   &\lesssim  N \int_{\mathbb{C}^{n}} \hat{G}_{\beta^{2}r}(w)
\|f(w)\|_{\mathcal{H}}|K(z,w)| e^{-2\phi(w)}dA(w)\nonumber\\
   &\lesssim \int_{\mathbb{C}^{n}} \hat{G}_{\delta}(w)
   \|f(w)\|_{\mathcal{H}}|K(z,w)| e^{-2\phi(w)}dA(w).
\end{align}
  Now, applying H\"older's inequality and \eqref{2.8}, we get
\begin{align*}\label{q1}
\|T_{G}f(z)\|_{\mathcal{H}}^{2}e^{-2\phi(z)}
   &\lesssim
   \left(\int_{\mathbb{C}^{n}} \hat{G}_{\delta}(w)
   \|f(w)\|_{\mathcal{H}}|K(z,w)| e^{-2\phi(w)}e^{-\phi(z)}dA(w) \right)^{2}\\
   &\lesssim
    \Bigg[
    \left(\int_{\mathbb{C}^{n}}  \hat{G}_{\delta}(w)
   \|f(w)\|_{\mathcal{H}}|K(z,w)|^{1/2} e^{-\phi(w)}e^{-\phi(w)/2}e^{-\phi(z)/2}\right)\\
   &\quad \quad \left(|K(z,w)|^{1/2}e^{-\phi(w)/2}e^{-\phi(z)/2} \right) dA(w)\Bigg]^{2}\\
    & \leq  \left(\int_{\mathbb{C}^{n}}  \hat{G}^{2}_{\delta}(w)
   \|f(w)\|^{2}_{\mathcal{H}}|K(z,w)| e^{-2\phi(w)}e^{-\phi(w)}e^{-\phi(z)}dA(w)\right)\\
    &\quad \quad \left(\int_{\mathbb{C}^{n}}|K(z,w)|e^{-\phi(w)}e^{-\phi(z)} dA(w)\right)\\ 
    &\lesssim \int_{\mathbb{C}^{n}}  \hat{G}^{2}_{\delta}(w)
   \|f(w)\|^{2}_{\mathcal{H}}|K(z,w)| e^{-2\phi(w)}e^{-\phi(w)}e^{-\phi(z)}dA(w).   
\end{align*}
Then, by Fubini's theorem and  using again \eqref{2.8}, we have 
\begin{align*}
    &\|T_{G}f\|^{2}_{2,\phi}=\int_{\mathbb{C}^{n}}\|T_{G}f(z)\|_{\mathcal{H}}^{2}e^{-2\phi(z)} dA(z)\\
    &\lesssim \int_{\mathbb{C}^{n}} \int_{\mathbb{C}^{n}}  \hat{G}^{2}_{\delta}(w)
   \|f(w)\|^{2}_{\mathcal{H}}|K(z,w)| e^{-2\phi(w)}e^{-\phi(w)}e^{-\phi(z)}dA(w)dA(z)\\
   &= \int_{\mathbb{C}^{n}}  \hat{G}^{2}_{\delta}(w)
   \|f(w)\|^{2}_{\mathcal{H}} e^{-2\phi(w)}e^{-\phi(w)}\int_{\mathbb{C}^{n}}|K(z,w)|e^{-\phi(z)}dA(z)dA(w)\\
   &\simeq \int_{\mathbb{C}^{n}}  \hat{G}^{2}_{\delta}(w)
   \|f(w)\|^{2}_{\mathcal{H}} e^{-2\phi(w)}dA(w),
\end{align*}
and we are done. Furthermore, by \eqref{ineq3.1}, \eqref{ineq3.2}, \eqref{ineq3.4}, and \eqref{ineq3.7}, one has \eqref{th11n}. 

To show that (iv) implies (v), note that using Lemma \ref{VectorValued-Lemma2.4Hicham}, \eqref{2.2}, and \eqref{2.6}, we obtain
\begin{equation}\label{finfty}
    \begin{split}
        &\int_{\mathbb{C}^{n}} \|f(z)\|^{2}_{\mathcal{H}}e^{-2\phi(z)}\|G(z)\|_{\mathcal{L}(\mathcal{H})}dA(z)\\
        &\lesssim \int_{\mathbb{C}^{n}}\left(
        \frac{1}{\rho(z)^{2n}}\int_{D^{\frac{\delta}{1+\delta}}(z)}\|f(w)\|^{2}_{\mathcal{H}}e^{-2\phi(w)}dA(w)
        \right) \|G(z)\|_{\mathcal{L}(\mathcal{H})}dA(z)\\
        &\leq \sum_{k=1}^{\infty}\int_{D^{\delta}(z_{k})}\frac{1}{\rho(z)^{2n}}\left(\int_{D^{\frac{\delta}{1+\delta}}(z)}\|f(w)\|^{2}_{\mathcal{H}}e^{-2\phi(w)}dA(w)
        \right)\|G(z)\|_{\mathcal{L}(\mathcal{H})}dA(z)\\
       & \lesssim \sum_{k=1}^{\infty}\frac{1}{\rho(z_{k})^{2n}} \int_{D^{\delta}(z_{k})}\left(\int_{D^{m\delta}(z_{k})}\|f(w)\|^{2}_{\mathcal{H}}e^{-2\phi(w)}dA(w)
        \right) \|G(z)\|_{\mathcal{L}(\mathcal{H})}dA(z)\\
        &=\sum_{k=1}^{\infty} \left(\int_{D^{m\delta}(z_{k})}\|f(w)\|^{2}_{\mathcal{H}}e^{-2\phi(w)}dA(w)
        \right) \hat{G}_{\delta}(z_{k})\\
       & \lesssim \|\{\hat{G}_{\delta}(z_{k})\}_{k}\|_{l^{\infty}}\|f\|^{2}_{2,\phi}.
       \end{split}
    \end{equation}

We finish the proof by showing that (v) implies (ii). Take $f=k_{z}e$ in \eqref{Carleson}. Then $\tilde{G}(z)\lesssim \|k_{z}\|^{2}_{2,\phi}=1$, implying that $\Tilde{G}(z)\in L^{\infty}(\mathbb{C}^{n},dA)$, and we are done with the proof.
\end{proof}

\begin{proof}[Proof of Theorem \ref{Thm1.2}]
One can show that (ii) implies (iii), and (iii) implies (iv) similarly as in the proof of Theorem \ref{Thm1.1}. 

Now, let us show that (iv) implies (ii). Let $\{z_{k}\}_{k}$ be a $\delta$-lattice. Assuming (iv) holds, for every $\epsilon>0$, there is $K\in\mathbb{N}$ such that whenever $k>K$, 
\begin{equation}\label{epsilon}
    \hat{G}_{\delta}(z_{k})<\epsilon.
\end{equation}
Let $m$ be defined as in (\ref{2.4}).
Then
\begin{align*}
    \Tilde{G}(z)
   &= \int_{\mathbb{C}^{n}} |k_{z}(w)|^{2} e^{-2\phi(w)} \|G(w)\|_{\mathcal{L}(\mathcal{H})} dA(w)\\
    &\leq
\int_{\cup_{k=1}^{K}\overline{D^{m\delta}(z_{k})}}|k_{z}(w)|^{2} e^{-2\phi(w)} \|G(w)\|_{\mathcal{L}(\mathcal{H})} dA(w)+\\
&\quad\quad\quad \quad\sum_{k=K+1}^{\infty} \rho(z_{k})^{2n}\hat{G}_{\delta}(z_{k})\left( \sup_{w\in D^{\delta}(z_{k})} |k_{z}(w)|^{2} e^{-2\phi(w)}
\right).
\end{align*}
 On one hand, by Lemma \ref{Lemma2.4-HichamToeplitz}, \eqref{epsilon} and \eqref{2.6}, we obtain
\begin{equation}\label{them12h}
\begin{split}
   &\sum_{k=K+1}^{\infty}\rho(z_{k})^{2n}\hat{G}_{\delta}(z_{k})\left( \sup_{w\in D^{\delta}(z_{k})} |k_{z}(w)|^{2} e^{-2\phi(w)}
\right)\\
&\lesssim  \sum_{k=K+1}^{\infty} \hat{G}_{\delta}(z_{k})  \left( \int_{D^{m\delta}(z_{k})}|k_{z}(w)|^{2} e^{-2\phi(w)}dA(w)
\right)\\
&\lesssim \sup_{k\geq K+1} \hat{G}_{\delta}(z_{k})  \left(\sum_{k=K+1}^{\infty} \int_{D^{m\delta}(z_{k})}|k_{z}(w)|^{2} e^{-2\phi(w)}dA(w)
\right)\\
&\lesssim \epsilon N \|k_{z}\|^{2}_{F^{2}_{\phi}(\mathbb{C}^{n})} \lesssim \epsilon.
\end{split}
\end{equation}
On the other hand, Lemma \ref{lem-Kernel estimates} implies that $k_{z}\to 0$ uniformly on $\cup_{k=1}^{K}\overline{D^{m\delta}(z_{k})}$ as $|z|\to \infty$. Therefore, 
\begin{equation}
    \int_{\cup_{k=1}^{K}\overline{D^{m\delta}(z_{k})}}|k_{z}(w)|^{2} e^{-2\phi(w)} \|G(w)\|_{\mathcal{L}(\mathcal{H})} dA(w)<\epsilon.
\end{equation}
This, together with \eqref{them12h}, proves (ii).

Next, we show that (i) implies (ii). Since $T_{G}$ is compact, Lemma \ref{WeakConvergenceNormalizedKernel} implies that $\|T_{G}k_{z}e\|_{2,\phi}$ converges to zero as $|z|\to\infty$. By \eqref{compactTG}, we have
\begin{align*}
    \Tilde{G}(z)
    \lesssim \|T_{G}k_{z}e\|_{2,\phi}\to 0,
\end{align*}
as $|z|\to \infty$.

To show that (iii) implies (i), let $\epsilon>0$ be arbitrary and $\{f_{j}\}_{j=1}^{\infty}$ be a sequence in $F^{2}_{\phi}(\mathcal{H})$ that converges to zero uniformly on any compact subset of $\mathbb{C}^{n}$. We want to show that for big enough $j\in\mathbb{N}$,
\begin{equation}
    \|T_{G}f_{j}\|_{2,\phi}\lesssim\epsilon.
\end{equation}
By our assumption, there is some $R>0$ such that
\begin{equation}\label{P2.4}
    \hat{G}_{\delta}(z) <\sqrt{\epsilon},
\end{equation}
whenever $|z|>R$.
 This together with  \eqref{star} and Lemma \ref{2.6}, we have 
\begin{equation}\label{Toeplitz weak convergence}
  \begin{split}
\|T_{G}f_{j}\|_{2,\phi}^{2}&\lesssim  
  \int_{|z|\leq R}  \hat{G}_{\delta}(w)^{2} \|f_{j}(w)\|^{2}_{\mathcal{H}} e^{-2\phi(w)}dA(w)\\
  &\quad\quad\quad\quad
  +\int_{|z|>R}  \hat{G}_{\delta}(w)^{2}  \|f_{j}(w)\|^{2}_{\mathcal{H}} e^{-2\phi(w)}dA(w)\\
  &\lesssim \epsilon + \epsilon \|f_{j}\|^{2}_{2,\phi}
  \lesssim \epsilon,
\end{split}
\end{equation}
where to bound the first integral we used the continuity of $\hat{G}_{\delta}(w)$ and the fact that $\{f_{j}\}_{j=1}^{\infty}$ converges to zero uniformly on any compact $\{|z|\leq R\}.$

To show that (iv) implies (v), let $\{f_k\}_{k}$ be a bounded sequence in $F^{2}_{\phi}(\mathcal{H})$ that converges to zero uniformly on compact subsets of $\mathbb{C}^n.$ By our assumption, letting $\varepsilon>0,$ there exists $r_0>0$ such that 
\begin{equation}
    \sup_{|z_k|> r_o}\hat{G}_{\delta}(z_{k}) <\varepsilon.
\end{equation}
Observe that there is $r_0 \le r_1$ such that if a point $z_k$ of the sequence $\{z_k\}_k$ belongs to
$\{|z| \le r_0\}$, then $D
^\delta(z_k) \subset \{|z| \le r_1\}$. So take $k$ big enough such that
$$
\sup_{ \{|z| \le r_1\}}\|f_k(z)\|_{\mathcal{H}}<\varepsilon.
$$
This together with \eqref{2.6} and \eqref{finfty}, we obtain
\begin{align*}
&\int_{\mathbb{C}^{n}} \|f_k(z)\|^{2}_{\mathcal{H}}e^{-2\phi(z)}\|G(z)\|_{\mathcal{L}(\mathcal{H})}dA(z) \\
&\lesssim \int_{\{|z| \le r_1\}} \|f_k(z)\|^{2}_{\mathcal{H}}e^{-2\phi(z)}\|G(z)\|_{\mathcal{L}(\mathcal{H})}dA(z)\\
&+\sup_{|z_k|> r_o}\hat{G}_{\delta}(z_{k}) \sum_{|z_k|>r_0} \int_{D^{m\delta}(z_{k})}\|f_k(w)\|^{2}_{\mathcal{H}}e^{-2\phi(w)}\|G(z)\|_{\mathcal{L}(\mathcal{H})}dA(w)
        \\
&\lesssim\varepsilon + \varepsilon\|f_k\|^{2}_{2,\phi}\lesssim \varepsilon,
\end{align*}
because $\{f_k\}_{k}$ belongs to  $F^{2}_{\phi}(\mathcal{H}).$ This implies that G satisfies a vanishing Carleson condition.

To finish the proof, we show that (v) implies (ii). Assuming (v), $I_{G}:F^{2}_{\phi}(\mathcal{H})\to L^{2}_{\phi}(\mathcal{H},\|G\|_{\mathcal{L}(\mathcal{H})}dA)$ is compact. Since $\{k_{z}e:z\in \mathbb{C}^{n}\}$ is bounded in $F^{2}_{\phi}(\mathcal{H})$ and $k_{z}e\to 0$ uniformly on compact subsets of $\mathbb{C}^{n}$, using (\ref{compact-Carleson}), we have
\begin{equation*}
    \int_{\mathbb{C}^{n}} |k_{z}(w)|^{2}e^{-2\phi(w)}\|G(w)\|_{\mathcal{L}(\mathcal{H})}dA(w)\to 0, \quad\textrm{ as } |z|\to \infty,
\end{equation*}
and therefore (ii) holds. This proves the desired result and completes the proof.
\end{proof}

\section{Schatten Class membership of Toeplitz operator}\label{Section4}
Here we give a proof of Theorem \ref{thm1.7}, characterizing the $p$-Schatten class membership of the vectorial Toeplitz operator $T_{G}$ acting on the Hilbert space $F^{2}_{\phi}(\mathcal{H})$. Before going through the proof, notice the following lemmas regarding an arbitrary $\delta$-lattice that are going to be useful in the proof of Theorem \ref{thm1.7}.
\begin{lem}\label{finitelattice1}
For $R>0$ and any finite sequence $\{z_{j}\}_{j=1}^{n}$ of different points in $\C^{n}$, let 
\begin{equation*}
    M_{R}(\{z_{j}\}_{j=1}^{n}):=\max_{1\leq j\leq n}\# \{k\in \{1,\cdots,n\}: |z_{j}-z_{k}|<R \min (\rho(z_{j},\rho(z_{k})))\}.
\end{equation*}
Then $\{z_{j}\}_{j=1}^{n}$ can be partitioned into at most $ M_{R}(\{z_{j}\}_{j=1}^{n})$ subsequences such that any two different points $z_{j}$ and $z_{k}$ in the same subsequence satisfy either $z_{j}\notin D^{R}(z_{k})$, or $z_{k}\notin D^{R}(z_{j})$. That is, $|z_{j}-z_{k}|\geq R \min (\rho(z_{j},\rho(z_{k})))$.
\end{lem}
  \begin{proof}
      The proof is identical to the doubling Fock spaces of the complex plane, as done in Lemma 6.8 in \cite{Oliverdoubling}.
  \end{proof}
  \begin{lem}\label{finitelyManySubsequences}
      Let $\delta>0$, $R>1$, and $\{z_{j}\}_{j\geq 1}$ be a $\delta$-lattice. Then $M_{R}(\{z_{j}\}_{j=1}^{m})\leq 6^{2n}R^{4n}\delta^{-2n}N_{\delta}$, for every finite sublattice $\{z_{j}\}_{j=1}^{m}$, where $N_{\delta}=\sup_{z\in\C^{n}}\sum_{j=1}^{\infty}\chi_{D^{\delta}(z_{j})}(z)$, as in (\ref{2.6}). 
  \end{lem}
  \begin{proof}
      The proof can be done similarly as in Lemma 6.9 in \cite{Oliverdoubling}.
  \end{proof}

\begin{proof}[Proof of Theorem \ref{thm1.7}] 
    Let $\{e_m^{z}\}_{m\geq 1}$ be any orthonormal basis of $\mathcal{H}$, depending on $z\in\C^{n}$. Notice that since $G(z)$ is a positive operator on $\mathcal{H}$, so are $\hat{G}^{op}_{\delta}(z)$ and $\tilde{G}^{op}(z)$. Moreover, $T_{G}$ is also a positive operator acting on $F^{2}_{\phi}(\mathcal{H})$.  We split the proof into two cases: For the first case, we choose $p\geq 1.$
    First, we show that (i) implies (ii). By the definition of the operator-valued Berezin transform, applying Proposition 1.31 in \cite{ZhuOperatorTheory}, Lemma \ref{lem-Kernel estimates}, and Lemma 5.1 in \cite{bigHankelBommierConstantine}, we obtain 
    \begin{equation}\label{tracecomputation}
    \begin{split}
&\int_{\mathbb{C}^n}\sum_{m=1}^{\infty}\Big( \langle \widetilde{G}^{op}(z)e_m^{z},e_m^{z}\rangle_{\mathcal{H}} \Big)^p\,\frac{dA(z)}{\rho(z)^{2n}}\\
     &= \int_{\mathbb{C}^n}\sum_{m=1}^{\infty}\Big( \int_{\mathbb{C}^n}\langle G(\zeta)e_m^{z},e_m^{z}\rangle_{\mathcal{H}} |k_{z}(\zeta)|^2\,e^{-2\phi(\zeta)} dA(\zeta)\Big)
     ^p\,\frac{dA(z)}{\rho(z)^{2n}}\\
      &= \int_{\mathbb{C}^n}\sum_{m=1}^{\infty}\Big( \int_{\mathbb{C}^n}\langle G(\zeta)k_{z}(\zeta)e_m^{z},k_{z}(\zeta)e_m^{z}\rangle_{\mathcal{H}} \,e^{-2\phi(\zeta)} dA(\zeta)\Big)
     ^p\,\frac{dA(z)}{\rho(z)^{2n}}\\
     &= \int_{\mathbb{C}^n}\sum_{m=1}^{\infty}\Big( \int_{\mathbb{C}^n}\langle T_{G}k_{z}(\zeta)e_m^{z},k_{z}(\zeta)e_m^{z}\rangle_{\mathcal{H}} \,e^{-2\phi(\zeta)} dA(\zeta)\Big)
     ^p\,\frac{dA(z)}{\rho(z)^{2n}}\\
&=\int_{\mathbb{C}^n}\sum_{m=1}^{\infty}\Big( \langle T_{G}k_{z}e_m^{z},k_{z}e_m^{z}\rangle \Big)
     ^p\,\frac{dA(z)}{\rho(z)^{2n}}\\
     &\leq \int_{\mathbb{C}^n}\sum_{m=1}^{\infty} \langle T_{G}^{p}k_{z}e_m^{z},k_{z}e_m^{z}\rangle \frac{dA(z)}{\rho(z)^{2n}}\\
     &\simeq \int_{\mathbb{C}^n}\sum_{m=1}^{\infty} \langle T_{G}^{p}K_{z}e_m^{z},K_{z}e_m^{z}\rangle e^{-2\phi(z)}dA(z)\\
     &=\Tr(T^{p}_{G})<\infty.
     \end{split}
    \end{equation}
    
 Now we show that (ii) implies (iii). Indeed, using (\ref{2.10}),
\begin{align*}
  &\int_{\mathbb{C}^n}\sum_{m=1}^{\infty}\Big(\Big\langle \hat{G}^{op}_\delta(z)e_m^{z},e_m^{z}\Big\rangle_{\mathcal{H}} \Big)^p\,\frac{dA(z)}{\rho(z)^{2n}}\\
  &\simeq   \int_{\mathbb{C}^n} \sum_{m=1}^{\infty}\Big(\Big\langle \int_{D_\delta(z)} G(\zeta)\,e_m^{z} \frac{dA(\zeta)}{\rho(\zeta)^{2n}},e_m^{z}\Big\rangle_{\mathcal{H}} \Big)^p\,\frac{dA(z)}{\rho(z)^{2n}}\\
  &\simeq \int_{\mathbb{C}^n} \sum_{m=1}^{\infty}\Big(\Big\langle \int_{\mathbb{C}^n} G(\zeta)\,e_m^{z} |k_{z}(\zeta)|^2\,e^{-2\phi(\zeta)} dA(\zeta),e_m^{z}\Big\rangle_{\mathcal{H}} \Big)^p\,\frac{dA(z)}{\rho(z)^{2n}}\\
  &\simeq \int_{\mathbb{C}^n} \sum_{m=1}^{\infty}\Big(\Big\langle  \Tilde{G}^{op}(\zeta)\,e_m^{z},e_m^{z}\Big\rangle_{\mathcal{H}} \Big)^p\,\frac{dA(z)}{\rho(z)^{2n}},
\end{align*}
and we can see that (ii) and (iii) are equivalent.  
Moreover, the statement (iii) implies (i). Indeed, let $\{f_{k}\}_{k\geq 1}$ be an orthonomal basis of $F^{2}_{\phi}(\C^{n})$. Then  $\{f_{k}e_{m}^{z}\}_{k,m\geq 1}$ is an orthonomal basis of $F^{2}_{\phi}(\mathcal{H})$. By Lemma \ref{lem2.14} we see that $T_{\hat{G}^{op}_{\delta}}\in S_{p}(F^{2}_{\phi}(\mathcal{H}))$. Recall that for any $\zeta\in D^{\delta}(z)$, there exists some small enough $r>0$ such that $D^{r}(\zeta)\subset D^{\delta}(z)$, coming from the Hausdorff property of $\C^{n}$. Moreover, for any $m,k\geq 1$, applying Fubini's Theorem and Lemma \ref{Lemma2.4-HichamToeplitz}, we get  
\begin{align*}
    \langle T_{\hat{G}^{op}_{\delta}} f_{k}e_{m}^{z},f_{k}e_{m}^{z}\rangle &
    =\int_{\C^{n}} \langle T_{\hat{G}^{op}_{\delta}} f_{k}(z)e_{m}^{z},f_{k}(z)e_{m}^{z}\rangle_{\mathcal{H}}e^{-2\phi(z)}dA(z)\\
    &=\int_{\C^{n}} \langle \hat{G}^{op}_{\delta}(z) f_{k}(z)e_{m}^{z},f_{k}(z)e_{m}^{z}\rangle_{\mathcal{H}}e^{-2\phi(z)}dA(z)\\
    &\simeq \int_{\C^{n}}\frac{1}{\rho(z)^{2n}}\int_{D^{\delta}(z)} \langle G(\zeta) e_{m}^{z},e_{m}^{z}\rangle_{\mathcal{H}}|f_{k}(z)|^{2}e^{-2\phi(z)}dA(\zeta)dA(z)\\
    &\gtrsim \int_{\C^{n}}\frac{1}{\rho(\zeta)^{2n}}\int_{D^{r}(\zeta)} \langle G(\zeta) e_{m}^{z},e_{m}^{z}\rangle_{\mathcal{H}}|f_{k}(z)|^{2}e^{-2\phi(z)}dA(z)dA(\zeta)\\
   & \gtrsim \int_{\C^{n}} \langle G(\zeta) e_{m}^{z},e_{m}^{z}\rangle_{\mathcal{H}}|f_{k}(\zeta)|^{2}e^{-2\phi(\zeta)}dA(\zeta)\\
   & =\int_{\C^{n}} \langle G(\zeta)f_{k}(\zeta) e_{m}^{z},f_{k}(\zeta)e_{m}^{z}\rangle_{\mathcal{H}}e^{-2\phi(\zeta)}dA(\zeta)\\
   &=\langle T_{G}f_{k}e_{m}^{z},f_{k}e_{m}^{z}\rangle.
\end{align*}
Since $T_{\hat{G}^{op}_{\delta}}\in S_{p}(F^{2}_{\phi}(\mathcal{H}))$, by definition $\sum_{k,m=1}^{\infty}(\langle T_{\hat{G}^{op}_{\delta}} f_{k}e_{m}^{z},f_{k}e_{m}^{z}\rangle)^{p}<\infty$, implying that 
\begin{equation*}
\sum_{k,m=1}^{\infty}(\langle T_{G}f_{k}e_{m}^{z},f_{k}e_{m}^{z}\rangle)^{p}<\infty.
\end{equation*}
Thus $T_{G}\in S_{p}(F^{2}_{\phi}(\mathcal{H}))$.

To finish the proof, it remains to show that the statements (iii) and (iv) are equivalent. Suppose that (iii) holds and let $\{z_j\}_{j\geq 1}$ be a $\delta$-lattice. Let $z\in D^{\delta}(z_{j})$. Then by (\ref{2.4}), there is some $c>1$ such that $D^{\delta}(z_{j})\subset D^{c\delta}(z)$. Then by definition it is easy to see that $\hat{G}^{op}_{\delta}(z_{j})\lesssim \hat{G}^{op}_{c\delta}(z)$. Hence, using (\ref{2.2}) and (\ref{2.6}), 
\begin{align*}
\sum_{j,m=1}^{\infty}\left(\langle \hat{G}^{op}_{\delta}(z_j)e_{m}^{z},e_{m}^{z}\rangle_{\mathcal{H}} \right)^{p}&=\sum_{j,m=1}^{\infty}\frac{\rho(z_{j})^{2n}}{\rho(z_{j})^{2n}} \left(\langle \hat{G}^{op}_{\delta}(z_j)e_{m}^{z},e_{m}^{z}\rangle_{\mathcal{H}} \right)^{p}\\
&\simeq\sum_{j,m=1}^{\infty}\int_{D^{\delta}(z_{j})}\left(\langle \hat{G}^{op}_{\delta}(z_j)e_{m}^{z},e_{m}^{z}\rangle_{\mathcal{H}} \right)^{p}\frac{dA(z)}{\rho(z_{j})^{2n}}\\
&\simeq\sum_{j,m=1}^{\infty}\int_{D^{\delta}(z_{j})}\left(\langle \hat{G}^{op}_{c\delta}(z)e_{m}^{z},e_{m}^{z}\rangle_{\mathcal{H}} \right)^{p}\frac{dA(z)}{\rho(z)^{2n}}\\
&=\sum_{m=1}^{\infty}\int_{\C^{n}}\sum_{j=1}^{\infty}\chi_{D^{\delta}(z_{j})}(z)\left(\langle \hat{G}^{op}_{c\delta}(z)e_{m}^{z},e_{m}^{z}\rangle_{\mathcal{H}} \right)^{p}\frac{dA(z)}{\rho(z)^{2n}}\\
&\lesssim \sum_{m=1}^{\infty}\int_{\C^{n}}\left(\langle \hat{G}^{op}_{c\delta}(z)e_{m}^{z},e_{m}^{z}\rangle_{\mathcal{H}} \right)^{p}\frac{dA(z)}{\rho(z)^{2n}}.
\end{align*}
Conversely, since $\rho(z_j)\simeq \rho(z)$ for any $z\in D^\delta(z_j),$ and $D^{\delta}(z)\subset D^{c\delta}(z_{j})$ for some $c>1$, we have
\begin{align*}
&\int_{\C^{n}}\sum_{m=1}^{\infty}\left(\langle \hat{G}^{op}_{\delta}(z)e_{m}^{z},e_{m}^{z}\rangle_{\mathcal{H}} \right)^{p}\frac{dA(z)}{\rho(z)^{2n}}\\
 &\simeq\int_{\C^{n}}\sum_{m=1}^{\infty}\left(\frac{1}{\rho(z)^{2n}}\int_{D^{\delta}(z)}\langle G(\zeta)e_{m}^{z},e_{m}^{z}\rangle_{\mathcal{H}}dA(\zeta) \right)^{p}\frac{dA(z)}{\rho(z)^{2n}}\\
 &\lesssim \sum_{j=1}^{\infty}\int_{D^{\delta}(z_{j})}\sum_{m=1}^{\infty}\left(\frac{1}{\rho(z_j)^{2n}}\int_{D^{\delta}(z)}\langle G(\zeta)e_{m}^{z},e_{m}^{z}\rangle_{\mathcal{H}}dA(\zeta) \right)^{p}\frac{dA(z)}{\rho(z_j)^{2n}}\\
 &\leq \sum_{j,m=1}^{\infty}\int_{D^{\delta}(z_{j})}\left(\frac{1}{\rho(z_j)^{2n}}\int_{D^{c\delta}(z_j)}\langle G(\zeta)e_{m}^{z},e_{m}^{z}\rangle_{\mathcal{H}}dA(\zeta) \right)^{p}\frac{dA(z)}{\rho(z_j)^{2n}}\\
 &\simeq \sum_{j,m=1}^{\infty}\left(\langle \hat{G}^{op}_{c\delta}(z_j)e_{m}^{z},e_{m}^{z}\rangle_{\mathcal{H}} \right)^{p}.
\end{align*}
This completes the proof of the first case for any choice of orthonormal basis $\{e_{m}^{z}\}_{z\geq 1}$.

Now, we start the second case where $0<p\le 1.$ Similar to the previous case, we can see that (ii) implies (iii) and (iii) implies (iv). Now, we prove that (iv) implies (ii). First note that by Lemma \ref{Lemma2.4-HichamToeplitz}, (\ref{2.4}), and Fubini's Theorem, there is some $c>1$ such that
\begin{align*}
\langle \Tilde{G}^{op}(z)e_m^{z},e_m^{z}\rangle_{\mathcal{H}}& = \int_{\mathbb{C}^n}\langle G(\zeta)e_m^{z},e_m^{z}\rangle_{\mathcal{H}}|k_{z}(\zeta)|^2 e^{-2\phi(\zeta)} dA(\zeta)\\
   &\lesssim \int_{\mathbb{C}^n}\langle G(\zeta)e_m^{z},e_m^{z}\rangle_{\mathcal{H}}\left(\int_{D^\delta(\zeta)} |k_{z}(\xi)|^2 e^{-2\phi(\xi)} \frac{dA(\xi)}{\rho(\zeta)^{2n}}\right) dA(\zeta)\\
   &\lesssim \int_{\mathbb{C}^n}\left(\langle \int_{D^{c\delta}(\xi)} G(\zeta)e_m^{z} \frac{dA(\zeta)}{\rho(\zeta)^{2n}},e_m^{z}\rangle_{\mathcal{H}}\right) |k_{z}(\xi)|^2 e^{-2\phi(\xi)} dA(\xi)  \\
    &\simeq \int_{\mathbb{C}^n}\langle  \hat{G}^{op}_{c\delta}(\xi)e_m^{z} ,e_m^{z}\rangle_{\mathcal{H}} |k_{z}(\xi)|^2 e^{-2\phi(\xi)} dA(\xi) .
\end{align*}
Since $p\le 1$, $(x+y)^{p}\leq x^{p}+y^{p}$. Moreover, whenever $\xi \in D^{\delta}(z_{j})$, $D^{c\delta
}(\xi)\subset D^{r}(z_{j})$ for some large enough $r>0$ and $\rho(\xi)\simeq \rho(z_{j})$. Hence, it is easy to see that $\hat{G}^{op}_{c\delta}(\xi)\lesssim \hat{G}^{op}_{r}(z_{j})$. Thus we obtain the following. 
\begin{align*}
   &\int_{\mathbb{C}^n}\sum_{m=1}^{\infty} \left(\langle \Tilde{G}^{op}(z)e_m^{z},e_m^{z}\rangle_{\mathcal{H}} \right)^p dA(z)\\
   &\lesssim \int_{\mathbb{C}^n}\sum_{m=1}^{\infty}\left(\int_{\mathbb{C}^n}\langle  \hat{G}^{op}_{c\delta}(\xi)e_m^{z} ,e_m^{z}\rangle_{\mathcal{H}} |k_{z}(\xi)|^2 e^{-2\phi(\xi)} dA(\xi)\right)^p\frac{dA(z)}{\rho(z)^{2n}}\\
   &\leq \sum_{j,m=1}^{\infty}\int_{\C^{n}} \left(\int_{D^{\delta}(z_{j})}\langle  \hat{G}^{op}_{c\delta}(\xi)e_m^{z} ,e_m^{z}\rangle_{\mathcal{H}} |k_{z}(\xi)|^2 e^{-2\phi(\xi)} dA(\xi)\right)^p\frac{dA(z)}{\rho(z)^{2n}}\\
    &\leq \sum_{j,m=1}^{\infty}\left(\langle  \hat{G}^{op}_{r}(z_{j})e_m^{z} ,e_m^{z}\rangle_{\mathcal{H}}\right)^{p}\int_{\C^{n}} \left(\sup_{\xi\in D^{\delta}(z_{j})} |k_{z}(\xi)|^2 e^{-2\phi(\xi)} \rho(z_{j})^{2n}\right)^p\frac{dA(z)}{\rho(z)^{2n}}.
\end{align*}
Using the fact that  
$
|k_{z}(\xi)|^2\,e^{-2\phi(\xi)}\rho(\xi)^{2n} \lesssim \exp\left(-\varepsilon d_\rho(z,\xi)\right),
$ for any $\varepsilon >0$, and Lemma 2 in \cite{lv2017bergman}, we get
\begin{align*}
\int_{\mathbb{C}^n}\sum_{m=1}^{\infty} \left(\langle \Tilde{G}^{op}(z)e_m^{z},e_m^{z}\rangle_{\mathcal{H}} \right)^p dA(z)
 &\lesssim   \sum_{j,m=1}^{\infty}\left(\langle  \hat{G}^{op}_{r}(z_{j})e_m^{z} ,e_m^{z}\rangle_{\mathcal{H}}\right)^{p} \int_{\C^{n}}e^{-p\epsilon d_{\rho}(z,z_{j})}\frac{dA(z)}{\rho(z)^{2n}}\\
&\lesssim\sum_{j,m=1}^{\infty}\left(\langle  \hat{G}^{op}_{r}(z_{j})e_m^{z} ,e_m^{z}\rangle_{\mathcal{H}}\right)^{p} \rho(z_{j})^{2n-2n}
 \simeq \sum_{j,m=1}^{\infty}\left(\langle  \hat{G}^{op}_{r}(z_{j})e_m^{z} ,e_m^{z}\rangle_{\mathcal{H}}\right)^{p}.
\end{align*}

We have just proved that (ii), (iii), and (iv) are equivalent. Now, it remains to prove that (ii) implies (i), and (i) implies (iv).
Suppose that (ii) is true. We know that $T_G\in S_p$ if and only if $T^p_G\in S_1.$ By Lemma 5.1 in \cite{bigHankelBommierConstantine} and (\ref{tracecomputation}), this is equivalent to 
$$\Tr(T^p_G)\simeq \int_{\mathbb{C}^n}\sum_{m=1}^{\infty}\langle T^p_Gk_{z}e_m^{z}, k_{z}e_m^{z}\rangle\, \frac{dA(z)}{\rho(z)^{2n}}<+\infty.$$
By Proposition 1.31 in \cite{ZhuOperatorTheory} for when $p\leq 1$, and since positivity of $G(z)$ implies positivity of $T_{G}$, we get 
$$ \langle T^p_Gk_{z}e_m, k_{z}e_m\rangle\lesssim \langle T_Gk_{z}e_m, k_{z}e_m\rangle^p.$$
This together with (\ref{tracecomputation}) implies that
\begin{align*}
  \Tr(T^p_G)\lesssim \int_{\mathbb{C}^n}\sum_{m=1}^{\infty}\left(\langle T_Gk_{z}e_m^{z}, k_{z}e_m^{z}\rangle\right)^{p} \frac{dA(z)}{\rho(z)^{2n}}
=\int_{\mathbb{C}^n}\sum_{m=1}^{\infty}\Big( \langle \widetilde{G}^{op}(z)e_m^{z},e_m^{z}\rangle_{\mathcal{H}} \Big)^p\,\frac{dA(z)}{\rho(z)^{2n}}<\infty,
\end{align*}
and we are done. 

Finally, we finish the proof by showing that (i) implies (iv). Notice that we have not yet made any specific assumption on the basis $\{e_{m}^{z}\}_{m\geq 1}$, and everything we have done so far during the proof regards any generic basis of $\mathcal{H}$. But now we will see how we are forced to restrict ourselves to a specific basis of $\mathcal{H}$. Note that the idea of the proof originally comes from the work of Luecking \cite{Luecking} in studying the Schatten class Toeplitz operators on Hardy and Bergman spaces. Later, an analogous idea was applied to studying the Schatten class Toeplitz operators with measure symbols on doubling Fock spaces in \cite{Oliverdoubling}. This approach has also been used to study the Schatten class Hankel operators acting on generalized Fock spaces in \cite{Ghazaleh,JaniSchattenClass}.

Assume that $0<\delta<1/2$, $R>1$, and fix $M\in\mathbb{N}$. Let $\{z_{j}\}_{j=1}^{M}$ be the finite sublattice obtained by considering the first $M$ elements of $\{z_{j}\}_{j\geq 1}$. Then Lemma \ref{finitelattice1} implies that $\{z_{j}\}_{j=1}^{M}$ can be partitioned into $M_{R}(\{z_{j}\}_{j=1}^{M})$ subsequences such that any two different points $a_{j}$ and $a_{k}$ in the same subsequence satisfy $|a_{j}-a_{k}|\geq R\min(\rho(a_{j}),\rho(a_{k}))$. Let $\{a_{j}\}_{j=1}^{s}$ be one such subsequence. Let $\{B_{k,m}(z)= f_k(z)e_m^{z}\}_{k,m\geq 1}$ be an orthonormal basis of $F^{2}_{\phi}(\mathcal{H})$ with $\{f_k\}_{k\geq 1}$ being an orthonormal basis of $F^{2}_{\phi}(\mathbb{C}^{n}).$ We consider a bounded linear operator $A$ on $F^{2}_{\phi}(\mathcal{H})$ by $A f(z) =\sum_{k,m=1}^{s}\langle f, B_{k,m}\rangle k_{a_{k}}(z)e_{m}^{z}$. Indeed, let $f,g\in F^{2}_{\phi}(\mathcal{H})$. Then by the Cauchy-Schwarz inequality, (\ref{2.8}), Lemma \ref{Lemma2.3}, Lemma \ref{VectorValued-Lemma2.4Hicham}, and (\ref{2.6}), we obtain
\begin{align*}
    |\langle Af,g\rangle|&\leq \|f\|_{2,\phi}\left( \sum_{k,m=1}^{s}|\langle k_{a_{k}}e_{m},g\rangle|^{2}\right)^{1/2} \\
    &\simeq\|f\|_{2,\phi}\left( \sum_{k,m=1}^{s}\left|\int_{\C^{n}}\rho(a_{k})^{n}e^{-\phi(a_{k})}\langle e_{m},g(z)\rangle_{\mathcal{H}}K(z,a_{k})e^{-2\phi(z)}dA(z)\right|^{2}\right)^{1/2}\\
    &=\|f\|_{2,\phi}\left( \sum_{k,m=1}^{s}\rho(a_{k})^{2n}e^{-2\phi(a_{k})}|\langle e_{m},g(a_{k})\rangle_{\mathcal{H}}|^{2}\right)^{1/2}\\
    &\lesssim \|f\|_{2,\phi}\left( \sum_{k,m=1}^{s}\int_{D^{\delta}(a_{k})}|\langle e_{m},g(w)\rangle_{\mathcal{H}}|^{2}e^{-2\phi(w)}dA(w)\right)^{1/2}\\
    &\leq \|f\|_{2,\phi}\left( \sum_{k=1}^{s}\int_{D^{\delta}(a_{k})}\|g(w)\|^{2}_{\mathcal{H}}|^{2}e^{-2\phi(w)}dA(w)\right)^{1/2}
    \leq N^{1/2}\|f\|_{2,\phi}\|g\|_{2,\phi},
\end{align*}
where the constants do not depend on $s$. Hence, $A$ is bounded.
Moreover, let $U(z)=\sum_{j=1}^{s}G\circ M_{z}\chi_{D^{\delta}(a_{j})}(z)$. Then $U\leq N G$, where $N$ is an in (\ref{2.6}). Since $T_G\in S_p,$ we can conclude that $T_{U}\in S_{p}$, and $\|T_U\|_{S_p}\leq N \|T_G\|_{S_p} .$ Set $T=A^*T_U A$ such that  $\|T\|_{S_p}\lesssim \|T_U\|_{S_p}$. It is easy to see that when $k,m>s$, $\langle TB_{k,m},B_{k,m}\rangle=\langle T_{U}AB_{k,m},AB_{k,m}\rangle=0$. We can split $T$ as $T=D_{s}+M_{s},$ where $D_{s}$ is the diagonal operator defined by
\begin{equation}\label{diagElement}
    D_{s}f=\sum_{k,m=1}^{s}  \langle TB_{k,m}, B_{k,m}\rangle \langle f,B_{k,m}\rangle B_{k,m}, \quad \textrm{ where } f\in F^{2}_{\phi}(\mathcal{H}),
\end{equation}
and $M_{s}$ is the off-diagonal operator defined by
\begin{equation}\label{offdiagElement}
\begin{split}
    M_{s}f
    &=\sum_{k,m=1}^{s} \sum_{\substack{r,n=1 \\ r\neq k, m\neq n}}^{s}  \langle TB_{k,m}, B_{r,n}\rangle\langle f,B_{k,m}\rangle B_{r,n}
    + \sum_{k,m=1}^{s} \sum_{\substack{r=1 \\ r\neq k, }}^{s}  \langle TB_{k,m}, B_{r,m}\rangle\langle f,B_{k,m}\rangle B_{r,m}\\
    & +\sum_{k,m=1}^{s} \sum_{\substack{n=1 \\ m\neq n,}}^{s}  \langle TB_{k,m}, B_{k,n}\rangle\langle f,B_{k,m}\rangle B_{k,n}, \quad \textrm{ where } f\in F^{2}_{\phi}(\mathcal{H}).
    \end{split}
\end{equation}
Recall that $U(z)=0$ if $z\notin \cup_{j=1}^{s}D^{\delta}(a_{j})$. Then using (\ref{2.10}), and positivity of $G(z)$ and $\hat{G}^{op}_{\delta}(z)$, there is a constant $C_{1}>0$, only depending on $\delta$ such that
\begin{equation}\label{Dp}
\begin{split}
    \|D_{s}\|^{p}_{S_{p}}&= \sum_{i,j=1}^{s}\left| \langle D_s B_{i,j },B_{i,j}\rangle\right|^{p}
    = \sum_{m=1}^{s}\sum_{k=1}^{s} \left| \langle T B_{k,m},B_{k,m}\rangle\right|^{p}\\
    &= \sum_{m=1}^{s}\sum_{k=1}^{s} \left| \int_{\mathbb{C}^{n}}\langle T B_{k,m}(z),B_{k,m}(z)\rangle_{\mathcal{H}}e^{-2\phi(z)}dA(z)\right|^{p}\\
    &= \sum_{m=1}^{s}\sum_{k=1}^{s} \left| \int_{\mathbb{C}^{n}}\langle T_{U} k_{a_{k}}(z)e_{m}^{z},k_{a_{k}}(z)e_{m}^{z}\rangle_{\mathcal{H}}e^{-2\phi(z)}dA(z)\right|^{p}\\
    &=\sum_{m=1}^{s}\sum_{k=1}^{s} \left| \int_{\mathbb{C}^{n}}\langle U(z) k_{a_{k}}(z)e_{m}^{z},k_{a_{k}}(z)e_{m}^{z}\rangle_{\mathcal{H}}e^{-2\phi(z)}dA(z)\right|^{p}\\
    &\geq \sum_{m=1}^{s}\sum_{k=1}^{s} \left| \int_{D^{\delta}(a_{k})}\langle G(z) e_{m}^{z},e_{m}^{z}\rangle_{\mathcal{H}} 
|k_{a_{k}}(z)|^{2} e^{-2\phi(z)}dA(z)\right|^{p}\\
&\geq C_1 \sum_{m=1}^{s}\sum_{k=1}^{s} \left| \int_{D^{\delta}(a_{k})}\langle G(z) e_{m}^{z},e_{m}^{z}\rangle_{\mathcal{H}} \frac{dA(z)}{\rho(z)^{2n}}\right|^{p}
= C_1 \sum_{m=1}^{s}\sum_{k=1}^{s} \left( \langle \hat{G}^{op}_{\delta}(a_{k}) e_{m}^{z},e_{m}^{z}\rangle_{\mathcal{H}} \right)^{p}.
\end{split}
\end{equation}
On the other hand, by Proposition 1.29 in \cite{ZhuOperatorTheory}, and using the fact that $(x+y)^{p}\leq x^{p}+y^{p}$ for $p\leq 1$, we obtain
\begin{equation}\label{offdiag1}
  \begin{split}
\|M_{s}\|^{p}_{S_{p}}&\leq\sum_{s,q=1}^{s} \sum_{i,j=1}^{s}\left| \langle M_s B_{s,q},B_{i,j}\rangle\right|^{p}\\
    &= \sum_{\substack{r,k=1\\ r\neq k}}^{s} \sum_{\substack{m,n=1\\ m\neq n}}^{s}\left| \langle T B_{k,m},B_{r,n}\rangle\right|^{p}+
     \sum_{\substack{k,r=1\\ r\neq k}}^{s}\sum_{m=1}^{s}\left| \langle T B_{k,m},B_{r,m}\rangle\right|^{p}
     +
\sum_{k=1}^{s} \sum_{\substack{n=1\\  m\neq n}}^{s}\left| \langle T B_{k,m},B_{k,n}\rangle\right|^{p}\\
&=\sum_{\substack{r,k=1\\ r\neq k}}^{s} \sum_{\substack{m,n=1\\ m\neq n}}^{s} \left| \int_{\mathbb{C}^{n}}\langle U(z) k_{a_{k}}(z)e_{m}^{z},k_{a_{r}}(z)e_{n}^{z}\rangle_{\mathcal{H}}e^{-2\phi(z)}dA(z)\right|^{p}\\
&\quad\quad+
    \sum_{\substack{k,r=1\\ r\neq k}}^{s}\sum_{m=1}^{s}\left| \int_{\mathbb{C}^{n}}\langle U(z) k_{a_{k}}(z)e_{m}^{z},k_{a_{r}}(z)e_{m}^{z}\rangle_{\mathcal{H}}e^{-2\phi(z)}dA(z)\right|^{p}\\
    &\quad\quad+
\sum_{k=1}^{s} \sum_{\substack{n=1\\  m\neq n}}^{s}\left| \int_{\mathbb{C}^{n}}\langle U(z) k_{a_{k}}(z)e_{m}^{z},k_{a_{k}}(z)e_{n}^{z}\rangle_{\mathcal{H}}e^{-2\phi(z)}dA(z)\right|^{p}.
\end{split}  
\end{equation}
Then by the definition of $U(z)$, and since $G(z)$ is a positive operator, we can write (\ref{offdiag1}) as follows. Note that positivity of $G(z)$ implies that $\langle G(z) e_{m}^{z},e_{m}^{z}\rangle_{\mathcal{H}}\geq 0$, but $\langle G(z) e_{m}^{z},e_{n}^{z}\rangle_{\mathcal{H}}$ is  a complex number.
\begin{equation}\label{offdiag2}
    \begin{split}
\|M_{s}\|^{p}_{S_{p}}& \leq N^p  
\sum_{\substack{r,k=1\\ r\neq k}}^{s} \sum_{\substack{m,n=1\\ m\neq n}}^{s}\left|\sum_{j=1}^{s} \int_{D^{\delta}(z_{j})}\langle G(z) e_{m}^{z},e_{n}^{z}\rangle_{\mathcal{H}}|k_{a_{k}}(z)||k_{a_{r}}(z)|e^{-2\phi(z)}dA(z)\right|^{p}\\
&\quad\quad+ N^p
    \sum_{\substack{k,r=1\\ r\neq k}}^{s}\sum_{m=1}^{s}\left( \sum_{j=1}^{s} \int_{D^{\delta}(z_{j})}\langle G(z) e_{m}^{z},e_{m}^{z}\rangle_{\mathcal{H}}|k_{a_{k}}(z)||k_{a_{r}}(z)|e^{-2\phi(z)}dA(z)\right)^{p}\\
    &\quad\quad+ N^p
\sum_{k=1}^{s} \sum_{\substack{n=1\\  m\neq n}}^{s}\left| \sum_{j=1}^{s} \int_{D^{\delta}(z_{j})}\langle G(z) e_{m}^{z},e_{n}^{z}\rangle_{\mathcal{H}}|k_{a_{k}}(z)|^{2}e^{-2\phi(z)}dA(z)\right|^{p}.
    \end{split}
\end{equation}
Define 
\begin{equation*}
    J^{m,n}_{k,r}(G,s)= \sum_{j=1}^{s} \int_{D^{\delta}(a_{j})}\langle G(z) e_{m}^{z},e_{n}^{z}\rangle_{\mathcal{H}}|k_{a_{k}}(z)||k_{a_{r}}(z)|e^{-2\phi(z)}dA(z).
\end{equation*}
Since $k\neq r$, then $|a_{k}-a_{r}|\geq R\min(\rho(a_{k}),\rho(a_{r}))$. Thus for $z\in D^{\delta}(a_{j})$ it is easy to see that either
\begin{equation}\label{latticeMink}
    |z-a_{k}|\geq \tilde{R}\min(\rho(z),\rho(a_{k})),
\end{equation}
or
\begin{equation}\label{latticeMinr}
    |z-a_{j}|\geq \tilde{R}\min(\rho(z),\rho(a_{j})),
\end{equation}
where $\tilde{R}=\frac{R-1}{3}$. Indeed, since $k\neq r$, then either $j\neq k$ or $j\neq r$. Without loss of generality assume that $j\neq k$, and therefore  $|a_{j}-a_{k}|\geq R\min(\rho(a_{j}),\rho(a_{k}))$. By contradiction, assume that $ |z-a_{k}|< \tilde{R}\min(\rho(z),\rho(a_{k}))$. Then by (\ref{2.2}),
\begin{align*}
    |a_{j}-a_{k}|&\leq |a_{j}-z|+|z-a_{k}|< \delta\rho(a_{j})+\tilde{R}\min(\rho(z),\rho(a_{k}))\\
    &\leq \delta\rho(a_{j})+\tilde{R}\min((1+\delta)\rho(a_{j}),\rho(a_{k}))\\
     &\leq \delta\rho(a_{j})+(1+\delta)\tilde{R}\min(\rho(a_{j}),\rho(a_{k})).
\end{align*}
If $\rho(a_{j})\leq \rho(a_{k})$, then $|a_{j}-a_{k}|<(\delta+(1+\delta)\tilde{R})\rho(a_{j})<\frac{1+3\tilde{R}}{2}\rho(a_{j})=\frac{R}{2}\rho(a_{j})$, which is a contradiction. Similarly, if $\rho(a_{k})\leq \rho(a_{j})$, (\ref{2.1}) implies that
\begin{align*}
  |a_{j}-a_{k}|&<  \delta\rho(a_{j})+(1+\delta)\tilde{R}\min(\rho(a_{j}),\rho(a_{k}))\\
  &\leq \delta\rho(a_{k})+\delta |a_{j}-a_{k}|+(1+\delta)\tilde{R}\rho(a_{k}).
\end{align*}
Hence, $|a_{j}-a_{k}|<\frac{\delta+(1+\delta)\tilde{R}}{1-\delta}\rho(a_{k})<(1+3\tilde{R})\rho(a_{k})=R\rho(a_{k})$, resulting in a contradiction. Therefore we can conclude that either (\ref{latticeMink}) or (\ref{latticeMinr}) holds.

Using (\ref{pk}) and (\ref{2.8}), we can write
\begin{equation*}
    |k_{a_{k}}(z)| |k_{a_{r}}(z)|e^{-2\phi(z)}\lesssim \frac{e^{\frac{-\epsilon}{2}[d_{\rho}(z,a_{k})+d_{\rho}(z,a_{r})]}}{\rho(z)^{n}}|k_{a_{k}}(z)|^{1/2} |k_{a_{r}}(z)|^{1/2}e^{-\phi(z)}.
\end{equation*}
Furthermore, since $z\in D^{\delta}(a_{j})$, and $r\neq k$, we can assume that $j\neq k$, and by the above argument we have $d_{\rho}(z,a_{k})+d_{\rho}(z,a_{r})\geq d_{\rho}(a_r,a_{k})\geq \tilde{R}$. Hence, for $z\in D^{\delta}(a_{j})$, we can conclude that $e^{\frac{-\epsilon}{2}[d_{\rho}(z,a_{k})+d_{\rho}(z,a_{r})]}\leq e^{\frac{-\epsilon}{2}\tilde{R}}$. Hence, using $\rho(z)\simeq \rho(a_{j})$, we obtain
\begin{equation*}
    J^{m,n}_{k,r}(G,s)\lesssim e^{\frac{-\epsilon}{2}\tilde{R}} \sum_{j=1}^{s} \frac{1}{\rho(a_{j})^{n}}\int_{D^{\delta}(a_{j})}\langle G(z) e_{m}^{z},e_{n}^{z}\rangle_{\mathcal{H}}|k_{a_{k}}(z)|^{1/2}|k_{a_{r}}(z)|^{1/2}e^{-\phi(z)}dA(z).
\end{equation*}
By Lemma \ref{Lemma2.4-HichamToeplitz} we have
\begin{equation*}
  |k_{a_{k}}(z)|^{1/2}e^{-\phi(z)/2}\lesssim \left(\frac{1}{\rho(z)^{2n}}\int_{D^{\delta}(z)}|k_{a_{k}}(\xi)|^{p/2}e^{-\frac{p}{2}\phi(\xi)}dA(\xi)\right)^{1/p}.  
\end{equation*}
Since $z\in D^{\delta}(a_{j})$, there exists some $m_{1}>1$ such that $D^{\delta}(z)\subset D^{m_{1}\delta}(a_{j})$. Therefore,
\begin{equation*}
  |k_{a_{k}}(z)|^{1/2}e^{-\phi(z)/2}\lesssim \frac{1}{\rho(z)^{2n/p}}\left(\int_{D^{m_{1}\delta}(a_{j})}|k_{a_{k}}(\xi)|^{p/2}e^{-\frac{p}{2}\phi(\xi)}dA(\xi)\right)^{1/p}=\frac{1}{\rho(z)^{2n/p}} S_{k}(a_{j})^{1/p},  
\end{equation*}
where
\begin{equation*}
    S_{k}(z)=\int_{D^{m_{1}\delta}(z)}|k_{a_{k}}(\xi)|^{p/2}e^{-\frac{p}{2}\phi(\xi)}dA(\xi).
\end{equation*}
Similarly,
\begin{equation*}
  |k_{a_{r}}(z)|^{1/2}e^{-\phi(z)/2}\lesssim \frac{1}{\rho(z)^{2n/p}} S_{r}(a_{j})^{1/p}.  
\end{equation*}
Therefore,
\begin{equation*}
    J^{m,n}_{k,r}(G,s)\lesssim e^{\frac{-\epsilon}{2}\tilde{R}} \sum_{j=1}^{s} \frac{1}{\rho(a_{j})^{n}}\int_{D^{\delta}(a_{j})}\langle G(z) e_{m}^{z},e_{n}^{z}\rangle_{\mathcal{H}}\frac{1}{\rho(z)^{4n/p}}S_{k}(a_{j})^{1/p} S_{r}(a_{j})^{1/p}dA(z).
\end{equation*}
Since $0<p<1$, $4/p-1>1$, and
\begin{align*}
    J^{m,n}_{k,r}(G,s)&\lesssim e^{\frac{-\epsilon}{2}\tilde{R}} \sum_{j=1}^{s} \frac{1}{\rho(a_{j})^{n}}\int_{D^{\delta}(a_{j})}\langle G(z) e_{m}^{z},e_{n}^{z}\rangle_{\mathcal{H}}\frac{1}{\rho(a_{j})^{4n/p}}S_{k}(a_{j})^{1/p} S_{r}(a_{j})^{1/p}dA(z)\\
    &\simeq e^{\frac{-\epsilon}{2}\tilde{R}} \sum_{j=1}^{s} \frac{1}{\rho(a_{j})^{n(\frac{4}{p}-1)}}S_{k}(a_{j})^{1/p} S_{r}(a_{j})^{1/p}\int_{D^{\delta}(a_{j})}\langle G(z) e_{m}^{z},e_{n}^{z}\rangle_{\mathcal{H}}\frac{dA(z)}{\rho(z)^{2n}}\\
    & \simeq e^{\frac{-\epsilon}{2}\tilde{R}} \sum_{j=1}^{s} \rho(a_{j})^{n(1-\frac{4}{p})}S_{k}(a_{j})^{1/p} S_{r}(a_{j})^{1/p}\langle \hat{G}^{op}_{\delta}(a_{j}) e_{m}^{z},e_{n}^{z}\rangle_{\mathcal{H}}.
\end{align*}
Then since $0<p<1$,
\begin{equation}\label{J}
   \left| J^{m,n}_{k,r}(G,s)\right|^{p}\lesssim e^{\frac{-p\epsilon}{2}\tilde{R}} \sum_{j=1}^{s} |\langle \hat{G}^{op}_{\delta}(a_{j}) e_{m}^{z},e_{n}^{z}\rangle_{\mathcal{H}}|^{p} \rho(a_{j})^{n(p-4)}S_{k}(a_{j}) S_{r}(a_{j}).
\end{equation}
Recall that
\begin{align*}
    \sum_{k=1}^{s}S_{k}(a_{j})&= \sum_{k=1}^{s}\int_{D^{m_{1}\delta}(a_{j})}|k_{a_{k}}(\xi)|^{p/2}e^{-\frac{p}{2}\phi(\xi)}dA(\xi)\\
    &=\int_{D^{m_{1}\delta}(a_{j})}\left(\sum_{k=1}^{s}|k_{a_{k}}(\xi)|^{p/2}\right)e^{-\frac{p}{2}\phi(\xi)}dA(\xi).
\end{align*}
Moreover, using (\ref{2.6}), (\ref{2.8}), (\ref{kernelIntegral}), and Lemma \ref{Lemma2.4-HichamToeplitz}, we can write
\begin{align*}
    \sum_{k=1}^{s}|k_{a_{k}}(\xi)|^{p/2}&=\sum_{k=1}^{s}\rho(a_{k})|K_{\xi}(a_{k})|^{p/2}e^{\frac{-p}{2}\phi(a_{k})}\\
    &\lesssim \sum_{k=1}^{s}\int_{D^{\delta}(a_{k})}|K_{\xi}(z)|^{p/2}e^{\frac{-p}{2}\phi(z)}\rho(z)^{n(\frac{p}{2}-2)}dA(z)\\
    &\leq N \int_{\C^{n}}|K_{\xi}(z)|^{p/2}e^{\frac{-p}{2}\phi(z)}\rho(z)^{n(\frac{p}{2}-2)}dA(z)
    \simeq N e^{\frac{p}{2}\phi(\xi)}\rho(\xi)^{\frac{-np}{2}}.
\end{align*}
Hence,
\begin{align*}
   \sum_{k=1}^{s}S_{k}(a_{j})&\lesssim N \int_{D^{m_{1}\delta}(a_{j})} \rho(\xi)^{\frac{-np}{2}}dA(\xi)\simeq \rho(a_{j})^{2n-\frac{np}{2}},
\end{align*}
where the constant only depends on $\delta$. Similarly,
\begin{align*}
   \sum_{r=1}^{s}S_{r}(a_{j})&\lesssim \rho(a_{j})^{2n-\frac{np}{2}},
\end{align*}
and this together with (\ref{J}) imply that
\begin{equation}\label{Jp}
     \sum_{\substack{k,r=1\\ r\neq k}}^{s} \left| J^{m,n}_{k,r}(G,s)\right|^{p}\lesssim e^{\frac{-p\epsilon}{2}\tilde{R}} \sum_{j=1}^{s} |\langle \hat{G}^{op}_{\delta}(a_{j}) e_{m}^{z},e_{n}^{z}\rangle_{\mathcal{H}}|^{p},
\end{equation}
where the constant only depends on $0<\delta<1/2$.

Since $T_{G}\in S_{p}(F^{2}_{\phi}(\mathcal{H}))$, it is in particular compact. Lemma \ref{compactnessOfG-hat} then implies that $\hat{G}^{op}_{\delta}(a_{j})$ is a compact operator on $\mathcal{H}$ for every $a_{j}\in \C^{n}$. Since $\hat{G}^{op}_{\delta}(a_{j})$ is positive, it is, in particular, self-adjoint. Then the spectral Theorem for self-adjoint and compact operators implies that there exists an orthonormal basis $\{e_{n}^{j}\}_{n\geq 1}$ of $\mathcal{H}$ consisting of eigenvectors of $\hat{G}^{op}_{\delta}(a_{j})$. That is,
\begin{equation*}
\mathcal{H}=\overline{\operatorname{span}\{e_{n}^{j}\}},\quad \textrm{where } n\geq 1.
\end{equation*}
Hence, $\langle \hat{G}^{op}_{\delta}(a_{j})e_{m}^{j},e_{n}^{j}\rangle_{\mathcal{H}}=0$, for $m\neq n$. 
Comparing (\ref{offdiag2}) with (\ref{J}), and by the positivity of $\hat{G}^{op}_{\delta}(a_{j})$, we can see that in the above basis
\begin{equation}\label{offdiagfinal}
    \begin{split}
\|M_s\|^{p}_{S_{p}}& \lesssim  
\sum_{k,m=1}^{s} \sum_{\substack{r,n=1\\ r\neq k,m\neq n}}^{s} \left|J^{m,n}_{k,r}(G,s)\right|^{p}+
    \sum_{k,m=1}^{s} \sum_{\substack{r=1\\ r\neq k}}^{s}\left( J^{m,m}_{k,r}(G,s)\right)^{p}+
\sum_{k,m=1}^{s} \sum_{\substack{n=1\\  m\neq n}}^{s}\left| J^{m,n}_{k,k}(G,s)\right|^{p}\\
&\lesssim  e^{\frac{-p\epsilon}{2}\tilde{R}} \sum_{\substack{m,n=1\\ m\neq n}}^{s}\sum_{j=1}^{s} |\langle \hat{G}^{op}_{\delta}(a_{j}) e_{m}^{j},e_{n}^{j}\rangle_{\mathcal{H}}|^{p}
+\lesssim e^{\frac{-p\epsilon}{2}\tilde{R}} \sum_{m=1}^{s}\sum_{j=1}^{s} |\langle \hat{G}^{op}_{\delta}(a_{j}) e_{m}^{j},e_{m}^{j}\rangle_{\mathcal{H}}|^{p}
\\&
\quad\quad\quad+\lesssim e^{\frac{-p\epsilon}{2}\tilde{R}} \sum_{\substack{m,n=1\\ m\neq n}}^{s}\sum_{j=1}^{s} |\langle \hat{G}^{op}_{\delta}(a_{j}) e_{m}^{j},e_{n}^{j}\rangle_{\mathcal{H}}|^{p}\\
&=e^{\frac{-p\epsilon}{2}\tilde{R}} \sum_{m=1}^{s}\sum_{j=1}^{s} \left(\langle \hat{G}^{op}_{\delta}(a_{j}) e_{m}^{j},e_{n}^{j}\rangle_{\mathcal{H}}\right)^{p},
    \end{split}
\end{equation}
where the first and the third terms vanish because of the compactness argument above. Therefore, by (\ref{Dp}) and (\ref{offdiagfinal}), we can conclude that
\begin{equation}\label{Ts}
\begin{split}
    \|T_{G}\|^{p}_{S_{p}}&\gtrsim \|T\|^{p}_{S_{p}}\geq \|D_{s}\|^p_{S_p}-\|M_{s}\|^p_{S_p}\\
    &\geq(C_{1}-Q(N)e^{\frac{-p\epsilon}{2}\tilde{R}}) \sum_{m,j=1}^{s}\left( \langle \hat{G}^{op}_{\delta}(z_{j}) e_{m}^{j},e_{m}^{j}\rangle_{\mathcal{H}}\right)^{p},
    \end{split}
\end{equation}
where $Q(N)$ is some power of $N$, not depending on $s$. Since $e^{\frac{-p\epsilon}{2}\tilde{R}}\to 0$ as $R\to\infty$, there is always a constant $R=R(p,\delta,N)$ such that $C(p,\delta,N)=C_{1}-Q(N)e^{\frac{-p\epsilon}{2}\tilde{R}}>0$.

Fix $M$ to be a positive integer. Then Lemma \ref{finitelyManySubsequences} implies that $\{z_{j}\}_{j=1}^{M}$ can be partitioned into at most $6^{2n}R^{4n}\delta^{-2n}N_{\delta}$ subsequences such that any different points $z_{j}$ and $z_{k}$ in the same subsequence satisfy $|z_{j}-z_{k}|\geq R \min(\rho(z_{j}),\rho(z_{k}))$. Then by (\ref{Ts}), we obtain
\begin{equation*}
    \sum_{m,j=1}^{M}\left( \langle \hat{G}^{op}_{\delta}(z_{j}) e_{m}^{j},e_{m}^{j}\rangle_{\mathcal{H}}\right)^{p}\leq C(p,\delta,N_{\delta})^{-1}
6^{2n}R^{4n}\delta^{-2n}N_{\delta}\|T_{G}\|^{p}_{S_{p}}<\infty.
\end{equation*}
Since the RHS does not depend on $M$, we are done with the proof of Theorem \ref{thm1.7}.
\end{proof}

\subsection*{Acknowledgments} The authors would like to thank Haakan Hedenmalm for fruitful discussions.

\printbibliography

\end{document}